\documentclass[a4paper,12pt]{amsart}
\usepackage{mathrsfs}
\usepackage{amsfonts}
\usepackage{amssymb}
\usepackage{ifthen}
\usepackage{graphicx}

\nonstopmode \numberwithin{equation}{section}
\usepackage[usenames]{color}
\newtheorem{thm}{Theorem}[section]
\newtheorem{lem}{Lemma}[section]
\newtheorem{cor}[thm]{Corollary}
\newtheorem{prop}[thm]{Proposition}

\newtheorem{step}{Step}[section]

\theoremstyle{definition}
\newtheorem{mlem}{Main lemma}[section]
\newtheorem{assertion}{Assertion}[section]
\newtheorem{cl}{Claim}[section]
\newtheorem{ca}{Case}[section]
\newtheorem{sca}{Subcase}[section]
\newtheorem{scl}{Subclaim}[section]
\newtheorem{conj}[thm]{Conjecture}
\newtheorem{fact}{Fact}[section]
\newtheorem{defn}[thm]{Definition}
\newtheorem{op}[thm]{Open Problem}
\newtheorem{ques}[thm]{Question}
\newtheorem{rem}[thm]{Remark}
\newtheorem{exam}[thm]{Example}

\numberwithin{equation}{section}

\newcounter {own}
\def\theown {\thesection       .\arabic{own}}

\newenvironment{pf}[1][]{%
 \vskip 3mm
 \noindent
 \ifthenelse{\equal{#1}{}}%
  {{\slshape Proof. }}%
  {{\slshape #1.} }%
 }%
{\qed\bigskip}

\newcounter{alphabet}
\newcounter{tmp}
\newenvironment{Thm}[1][]{\refstepcounter{alphabet}%
\bigskip%
\noindent%
{\bf Theorem \Alph{alphabet}}%
\ifthenelse{\equal{#1}{}}{}{ (#1)}%
{\bf .} \itshape}{\vskip 8pt}

\makeatletter
\newcommand{\Ref}[1]{\@ifundefined{r@#1}{}{\setcounter{tmp}{\ref{#1}}\Alph{tmp}}}
\makeatother

\newcounter{alphabet2}

\def\be{\begin{equation}}
\def\ee{\end{equation}}

\newcommand{\ben}{\begin{enumerate}}
\newcommand{\een}{\end{enumerate}}

\newcommand{\blem}{\begin{lem}}
\newcommand{\elem}{\end{lem}}
\newcommand{\bthm}{\begin{thm}}
\newcommand{\ethm}{\end{thm}}
\newcommand{\bcor}{\begin{cor}}
\newcommand{\ecor}{\end{cor}}
\newcommand{\beg}{\begin{exam}}
\newcommand{\eeg}{\end{exam}}
\newcommand{\begs}{\begin{examples}}
\newcommand{\eegs}{\end{examples}}
\newcommand{\bdefe}{\begin{defn}}
\newcommand{\edefe}{\end{defn}}
\newcommand{\bprob}{\begin{prob}}
\newcommand{\eprob}{\end{prob}}
\newcommand{\bques}{\begin{ques}}
\newcommand{\eques}{\end{ques}}
\newcommand{\bei}{\begin{itemize}}
\newcommand{\eei}{\end{itemize}}
\newcommand{\bcon}{\begin{conj}}
\newcommand{\econ}{\end{conj}}
\newcommand{\bop}{\begin{op}}
\newcommand{\eop}{\end{op}}

\newcommand{\bas}{\begin{assertion}}
\newcommand{\eas}{\end{assertion}}

\newcommand{\bfa}{\begin{fact}}
\newcommand{\efa}{\end{fact}}

\newcommand{\bca}{\begin{ca}}
\newcommand{\eca}{\end{ca}}

\newcommand{\bst}{\begin{step}}
\newcommand{\est}{\end{step}}

\newcommand{\bsca}{\begin{sca}}
\newcommand{\esca}{\end{sca}}

\newcommand{\bcl}{\begin{cl}}
\newcommand{\ecl}{\end{cl}}

\newcommand{\bmlem}{\begin{mlem}}
\newcommand{\emlem}{\end{mlem}}

\newcommand{\bscl}{\begin{scl}}
\newcommand{\escl}{\end{scl}}

\newcommand{\bcons}{\begin{conjs}}
\newcommand{\econs}{\end{conjs}}

\newcommand{\bprop}{\begin{prop}}
\newcommand{\eprop}{\end{prop}}

\newcommand{\br}{\begin{rem}}
\newcommand{\er}{\end{rem}}
\newcommand{\brs}{\begin{rems}}
\newcommand{\ers}{\end{rems}}
\newcommand{\bo}{\begin{obser}}
\newcommand{\eo}{\end{obser}}
\newcommand{\bos}{\begin{obsers}}
\newcommand{\eos}{\end{obsers}}
\newcommand{\bpf}{\begin{pf}}
\newcommand{\epf}{\end{pf}}
\newcommand{\ba}{\begin{array}}
\newcommand{\ea}{\end{array}}
\newcommand{\beq}{\begin{eqnarray}}
\newcommand{\beqq}{\begin{eqnarray*}}
\newcommand{\eeq}{\end{eqnarray}}
\newcommand{\eeqq}{\end{eqnarray*}}

\newcommand{\ds}{\displaystyle}

\newcounter{minutes}
\divide\time by 60
\newcounter{hours}
\multiply\time by 60 \addtocounter{minutes}{-\time}

\begin{document}

\bibliographystyle{amsplain}
\title []
{Some sharp Schwarz-Pick type estimates and their applications   of harmonic and pluriharmonic functions}

\def\thefootnote{}
\footnotetext{ \texttt{\tiny File:~\jobname .tex,
          printed: \number\day-\number\month-\number\year,
          \thehours.\ifnum\theminutes<10{0}\fi\theminutes}
} \makeatletter\def\thefootnote{\@arabic\c@footnote}\makeatother

\author{Shaolin Chen}
 \address{Sh. Chen, College of Mathematics and
Statistics, Hengyang Normal University, Hengyang, Hunan 421008,
People's Republic of China.} \email{mathechen@126.com}

\author{Hidetaka Hamada}
\address{H. Hamada, Faculty of Science and Engineering, Kyushu Sangyo University,
3-1 Matsukadai 2-Chome, Higashi-ku, Fukuoka 813-8503, Japan.}
\email{ h.hamada@ip.kyusan-u.ac.jp}



\begin{abstract}
The purpose of this paper is to study the Schwarz-Pick type inequalities for harmonic or pluriharmonic functions.
By analogy with the generalized Khavinson conjecture, we first
give some sharp  estimates of the norm of harmonic functions from the Euclidean unit ball
 in $\mathbb{R}^n$ into  the unit ball of the real Minkowski space.
Next,
we give several sharp Schwarz-Pick type inequalities for
pluriharmonic functions from the Euclidean unit ball in $\mathbb{C}^n$
or from the unit polydisc in $\mathbb{C}^n$
into the unit ball of the Minkowski space.
Furthermore,
we establish some sharp coefficient type
Schwarz-Pick  inequalities
for pluriharmonic functions defined in the Minkowski space. Finally, we use the obtained  Schwarz-Pick type inequalities to
discuss the Lipschitz continuity,  the Schwarz-Pick type lemmas of arbitrary order and the Bohr phenomenon of harmonic  or pluriharmonic
functions.
\end{abstract}

\maketitle \pagestyle{myheadings} \markboth{ Sh. Chen and H. Hamada}{Some sharp Schwarz-Pick type estimates  and their applications}

\section{Introduction}

Let $\mathbb{C}^{n}$ be the complex space  of dimension $n$, where
$n$ is a positive integer. We can also interpret $\mathbb{C}^{n}$ as the
real $2n$-space $\mathbb{R}^{2n}$. We use $\ell_{p}^{n}$ to denote  the Minkowski space  defined by
$\mathbb{C}^{n}$ together with the $p$-norm

$$ \|z\|_{p}:=\begin{cases}
\displaystyle \left(\sum_{j=1}^{n}|z_{j}|^{p}\right)^{1/p}, &\, p\in[1,\infty),\\
 \displaystyle \max_{1\leq j\leq n}|z_{j}|, &\,
p=\infty.
\end{cases}$$
It is well known that $\ell_{p}^{n}$ is a Banach space. For
$p\in[1,\infty]$, let
$\mathbb{B}_{\ell_{p}^{n}}:=\{z\in\mathbb{C}^{n}:~\|z\|_{p}<1\}$ and $\mathbf{B}_{\ell_{p}^{n}}:=\{x\in\mathbb{R}^{n}:~\|x\|_{p}<1\}$.
In particular, let $\mathbb{D}:=\mathbb{B}_{\ell_{p}^{1}}$.

The
classical Schwarz-Pick lemma states that an analytic function $f$
of $\mathbb{D}$ into itself satisfies

\begin{equation}\label{eq-sp}
|f'(z)|\leq\frac{1-|f(z)|^{2}}{1-|z|^{2}},~z\in\mathbb{D}.
\end{equation}
For complex-valued harmonic functions $f$ of $\mathbb{D}$ into itself, Colonna \cite{Co-1989} proved the following sharp Schwarz-Pick lemma:
\be\label{Col-89}\left|\frac{\partial f(z)}{\partial z}\right|+\left|\frac{\partial f(z)}{\partial \overline{z}}\right|\leq
\frac{4}{\pi}\frac{1}{1-|z|^{2}},~z\in\mathbb{D}.\ee

For real-valued harmonic functions
$u$ of $\mathbb{D}$
into $(-1,1)$,
there is the following Schwarz-Pick lemma in \cite[Theorem 6.26]{ABR}:
\be\label{2001}\|\nabla u(0)\|_{2}\leq\frac{4}{\pi}.\ee
Since the inequality (\ref{2001}) can be rewritten into the following form
for an analytic function $f$ in $\mathbb{D}$ (cf. \cite{M-1950}):
$$|f'(z)|\leq\frac{4}{\pi}\frac{1}{1-|z|^{2}}\sup_{w\in\mathbb{D}}|{\rm Re}(f(w))|,$$
where $\mbox{Re}(f)$ means
the real part of $f$,
it may be considered as a version of the Schwarz-Pick lemma for real valued harmonic functions.
Kalaj and Vuorinen \cite{K-1}  improved the classical inequality (\ref{2001}) into the following sharp form:
\be\label{KV-2011}\|\nabla u(z)\|_{2}\leq\frac{4}{\pi}\frac{1-|u(z)|^{2}}{1-|z|^{2}},~z\in\mathbb{D}.\ee

In the book of Protter and Weinberger \cite{PW}, there is the following estimate of the gradient of
a  harmonic function of a domain $\Omega\subset\mathbb{R}^{n}~(n\geq2)$ into $\mathbb{R}$:
\be\label{PW-1-1} \|\nabla u(x)\|_{2}\leq\frac{n\omega_{n-1}}{(n-1)\omega_{n}d_{\Omega}(x)}{\rm osc}_{\Omega}(u),\ee
where $\omega_{n}$ is the volume of $\mathbf{B}_{\ell_{2}^{n}}$, $\omega_{n-1}$ is the area of $\partial\mathbf{B}_{\ell_{2}^{n}}$, ${\rm osc}_{\Omega}(u)$
is the oscillation of $u$ in $\Omega$, and $d_{\Omega}(x)$ is the distance from
$x$ to the boundary $\partial\Omega$ of $\Omega$.
(\ref{PW-1-1}) is a consequence of the following inequality:

$$\|\nabla u(0)\|_{2}\leq\frac{2n\omega_{n-1}}{(n-1)\omega_{n}R}\sup_{\|y\|_{2}<R}|u(y)|.$$
We refer the reader to see \cite{KM} for more details.
For any fixed $x\in\mathbf{B}_{\ell_{2}^{n}}$, let $C(x)$ be the smallest number such that
the following inequality
$$\|\nabla u(x)\|_{2}\leq C(x)\sup_{y\in\mathbf{B}_{\ell_{2}^{n}}}|u(y)|$$ holds for all
bounded harmonic functions $u$ of $\mathbf{B}_{\ell_{2}^{n}}$ into $\mathbb{R}$.
Similarly, for $x\in\mathbf{B}_{\ell_{2}^{n}}$ and $\iota\in\partial\mathbf{B}_{\ell_{2}^{n}},$
denote by $C(x,\iota)$ the smallest number such that the inequality

$$|\langle\nabla u(x),\iota\rangle|\leq C(x,\iota)\sup_{y\in\mathbf{B}_{\ell_{2}^{n}}}|u(y)|$$
holds for all
bounded harmonic functions $u$ of $\mathbf{B}_{\ell_{2}^{n}}$ into $\mathbb{R}$,
where $\langle\cdot,\cdot\rangle$ is the Euclidean inner product on
$\mathbb{R}^{n}$.
Since $$\|\nabla u(x)\|_{2}=\sup_{\iota\in\partial\mathbf{B}_{\ell_{2}^{n}}}|\langle\nabla u(x),\iota\rangle|,$$
we see that $$C(x)=\sup_{\iota\in\partial\mathbf{B}_{\ell_{2}^{n}}}C(x,\iota).$$
In \cite[p. 171]{KM}, Kresin and Maz'ya posed the generalized Khavinson problem for bounded harmonic
functions of $\mathbf{B}_{\ell_{2}^{n}}$ into $\mathbb{R}$ as follows (see also \cite[p. 220]{Kh}, \cite[Conjecture 1]{Adv-19} and
\cite[Conjecture 1]{Liu-19}).

\begin{conj}\label{Conj-1} For $x\in\mathbf{B}_{\ell_{2}^{n}}\backslash\{0\}$, we have

$$C(x)=C(x,\mathbf{n}_{x}),$$ where
$\mathbf{n}_{x}:=x/\|x\|_2$ is the unit outward normal vector to the sphere $\|x\|_{2}\partial\mathbf{B}_{\ell_{2}^{n}}:=\{y\in\mathbb{R}^{n}:~\|y\|_{2}=\|x\|_{2}\}$
at $x$.
\end{conj}
In 2017, Markovi\'c \cite{Mar} proved this conjecture when $x$ is near the boundary of the unit ball.
Kalaj \cite{K-2017} showed that the conjecture is true for $n=4$. In 2019, Melentijevi\'c \cite{Adv-19}
 proved a result, which confirmed the conjecture for $n=3$.
See \cite[Chapter 6]{KM-2012} for solutions of various Khavinson-type extremal problems for harmonic functions
on $\mathbf{B}_{\ell_{2}^{n}}$ and on a half-space in $\mathbb{R}^{n}.$
    Very recently, Liu \cite{Liu-19} showed the following result, which confirmed the generalized Khavinson conjecture for  $n\geq3$.

\begin{Thm}{\rm (\cite[Theorem 2]{Liu-19})}\label{Liu-2019}
For $n\geq3$, if $u$ is a bounded harmonic function of $\mathbf{B}_{\ell_{2}^{n}}$ into
$\mathbb{R}$, then we have the following sharp inequality:

$$\|\nabla u(x)\|_{2}\leq
\frac{c_{n}}{1-\|x\|_{2}^{2}}\left\{\int_{-1}^{1}\frac{\left|t-\frac{n-2}{n}\|x\|_{2}\right|(1-t^{2})^{\frac{n-3}{2}}}
{(1-2t\|x\|_{2}+\|x\|_{2}^{2})^{\frac{n-2}{2}}}dt\right\}
\|u\|_{\infty},~x\in\mathbf{B}_{\ell_{2}^{n}},$$
where $c_{n}=\frac{2\Gamma\left(\frac{n+2}{2}\right)}{\Gamma\left(\frac{1}{2}\right)\Gamma\left(\frac{n-1}{2}\right)}$ and
$\Gamma(x)$ is the Gamma function.
\end{Thm}

For mappings with values in higher dimensional spaces,
Pavlovi\'c \cite{P-1,P-2019} showed that the inequality (\ref{eq-sp}) does not hold for analytic functions $f$ of
$\mathbb{D}$ into $\mathbb{B}_{\ell_{2}^{k}}$, where $k\geq2$ is an integer.
For example, the function $f(z)=(z,1)/\sqrt{2}$ for $z\in\mathbb{D}$ satisfies
$$\|f'(0)\|_{2}=\sqrt{1-\|f(0)\|_{2}^{2}}>1-\|f(0)\|_{2}^{2}.$$
However, Pavlovi\'c  proved the following Schwarz-Pick type lemma for analytic functions $f$ of
$\mathbb{D}$ into $\mathbb{B}_{\ell_{2}^{n}}$:
\[
|\nabla\|f(z)\|_{2}|\leq\frac{1-\|f(z)\|_{2}^{2}}{1-|z|^{2}},
\quad z\in\mathbb{D},
\]
where $\nabla\|f(z)\|_{2}$ denotes the  gradient of $\|f\|_{2}$.

In general, for $p\in[1,\infty]$, let $f$ be a  differentiable   function of $\mathbf{B}_{\ell_{2}^{n}}$
into $\mathbf{B}_{\ell_{p}^{k}}$, where  $k$ is a positive integer. Denote by

$$|\nabla \|f(z)\|_{p}|:=\limsup_{w\rightarrow z}\frac{|\|f(z)\|_{p}-\|f(w)\|_{p}|}{\|z-w\|_{2}}$$
the  gradient of $\|f\|_{p}$.

By analogy with the  generalized Khavinson Conjecture, it is nature to ask the following question
for mappings with values in higher dimensional spaces.

\begin{ques}\label{q-1}
For $p\in(1,\infty)$ and $n\geq3$, let $u$ be a harmonic function of $\mathbf{B}_{\ell_{2}^{n}}$  into $\mathbf{B}_{\ell_{p}^{\nu}}$,
where $\nu$ is a positive integer. Is there the sharp upper bound $C^{\ast}(x)$ such that
the following inequality
$$|\nabla \|u(x)\|_{p}|\leq C^{\ast}(x),~x\in\mathbf{B}_{\ell_{2}^{n}}$$ holds?
\end{ques}



 In \cite{Zhu}, Zhu considered Question \ref{q-1} for pluriharmonic functions (see Section \ref{csw-sec1} for the definition)
 and established a  Schwarz-Pick type estimate for pluriharmonic functions $f$ of
$\mathbb{B}_{\ell_{2}^{n}}$ into itself as follows.


\begin{Thm}{\rm (\cite[Theorem 1.1]{Zhu})}\label{Z-1}
For $n\geq1$, let $f$ be a pluriharmonic function of $\mathbb{B}_{\ell_{2}^{n}}$ into itself. Then the following inequality
$$|\nabla\|f(z)\|_{2}|\leq\frac{4\sqrt{n}}{\pi}\frac{1}{ (1-\|z\|_{2})}$$ holds for all $z\in\mathbb{B}_{\ell_{2}^{n}}.$
\end{Thm}

In \cite{Zhu}, Zhu also obtained the following Schwarz-Pick type lemma for
pluriharmonic functions of $\mathbb{B}_{\ell_{2}^{n}}$ into itself.

\begin{Thm}{\rm (\cite[Theorem 1.2]{Zhu})}\label{Z-2}
Let $f$ be a pluriharmonic function of $\mathbb{B}_{\ell_{2}^{n}}$ into itself.
Then
the following inequality
\[
\sum_{j=1}^{n}\left(\left|\frac{\partial f_{j}(z)}{\partial z_{k}}\right|^{2}+
\left|\frac{\partial f_{j}(z)}{\partial\overline{z}_{k}}\right|^{2}\right)
\leq\frac{1-\|f(z)\|_{2}^{2}}{(1-\|z\|_{2})^{2}},
\quad k=1,\dots, n
\]
holds for any $z\in \mathbb{B}_{\ell_{2}^{n}}$.
\end{Thm}


Let us recall the following classical result which we call the coefficient type Schwarz-Pick   lemma of analytic functions (cf. \cite{Neh}):
If $f$ is an analytic function of $\mathbb{D}$ into itself with $f(z)=\sum_{n=0}^{\infty}a_{n}z^{n},$ then for each $n\geq1$,
\be\label{Coe-1}|a_{n}|\leq1-|a_{0}|^{2}.\ee

By (\ref{Coe-1}), we have
\be\label{lf-1}|f'(0)|\leq1-|f(0)|^{2}.\ee
For any fixed $z\in\mathbb{D}$, let $F(w)=f(\phi(w))$,
where $\phi(w)=(z+w)/(1+\overline{z}w)$ for $w\in\mathbb{D}$. It follows from (\ref{lf-1}) that
$$|f'(z)|\left(1-|z|^{2}\right)=|F'(0)|\leq1-|F(0)|^{2}=1-|f(z)|^{2},$$
which implies that the
classical Schwarz-Pick lemma (\ref{eq-sp}) is a special case of (\ref{Coe-1}).

In \cite[Lemma 1]{CPW-2011},
the coefficient type Schwarz-Pick lemma of complex-valued harmonic functions
was established as follows:

\begin{Thm}\label{Coe-2}
If $f=h+\overline{g}$ is a complex-valued harmonic function of $\mathbb{D}$ into itself with $h(z)=\sum_{n=0}^{\infty}a_{n}z^{n}$ and $g(z)=\sum_{n=1}^{\infty}b_{n}z^{n}$, then $|a_{0}|\leq1$ and for all $n\geq1$,
\[
|a_{n}|+|b_{n}|\leq\frac{4}{\pi}.
\]
\end{Thm}
Applying the M\"obius transformation and Theorem \Ref{Coe-2}, we can easily obtain (\ref{Col-89}) (see \cite[Theorem 4]{CPW-2012}).


 In
1920, Sz\'asz \cite{S} extended the inequality (\ref{eq-sp}) to the
following estimate involving higher order derivatives:

\begin{equation}\label{eq-sz-1}
|f^{(2k+1)}(z)|\leq\frac{(2k+1)!}{(1-|z|^{2})^{2k+1}}\sum_{j=0}^{k}{k\choose
j}^{2}|z|^{2j},
\end{equation} where $k\in\{1,2,\ldots\}.$  Later,
Ruscheweyh (cf. \cite{AW,AW-09}) improved (\ref{eq-sz-1}) to the
following sharp form:

\begin{Thm}\label{Higher-1}
Let $f$ be an analytic function of $\mathbb{D}$ into itself.
Then
$$|f^{(k)}(z)|\leq\frac{k!(1-|f(z)|^{2})}{(1-|z|)^{k}(1+|z|)},
\quad z\in \mathbb{D},
$$
where $k\in\{1,2,\ldots\}.$
\end{Thm}
On the   Schwarz-Pick type estimates for derivatives of arbitrary order on bounded analytic functions of $\mathbb{B}_{\ell_{2}^{n}}$ into $\mathbb{C}$,
see \cite{DCP,LC}.

In 1914,  Bohr \cite{B} proved the following remarkable result on
power series in one complex variable:

\begin{Thm}{\rm (Bohr)}\label{Bohr-t3}
There exists $\rho\in(0,1)$ with the property that if a power series
$\sum_{k=0}^{\infty}a_{k}z^{k}$ converges in the unit disk and its
sum has modulus less than $1$, then
\[
\sum_{k=0}^{\infty}|a_{k}z^{k}|<1,
\quad  \mbox{ for all } |z|<\rho.
\]
\end{Thm}
Bohr's paper \cite{B}, compiled by   Hardy from correspondence,
indicates that Bohr initially proved the radius $\rho=1/6$, but this
was quickly improved to the sharp constant  $\rho=1/3$ by  Riesz,
Schur, and  Wiener, independently.

In the following, we write an
$n$-variable power series $\sum_{\alpha}a_{\alpha}z^{\alpha}$ using
the standard multi-index notation: $\alpha$ denotes an $n$-tuple
$(\alpha_{1},\ldots,\alpha_{n})$ of nonnegative integers, $|\alpha|$
denotes the sum $\sum_{k=1}^{n}\alpha_{k}$ of its components,
$\alpha!$ denotes the product $\prod_{k=1}^{n}\alpha_{k}!$ of the
factorials of its components, $z$ denotes an $n$-tuple
$(z_{1},\ldots,z_{n})$ of complex numbers, and $z^{\alpha}$ denotes
the product $\prod_{k=1}^{n}z_{k}^{\alpha_{k}}$.

For
$p\in[1,\infty]$, we
denote by $\mathscr{H}(\mathbb{B}_{\ell_{p}^{n}})$  the set of all
holomorphic functions of $\mathbb{B}_{\ell_{p}^{n}}$ into
$\mathbb{C}$.
Set
$\mathscr{H}_{1}(\mathbb{B}_{\ell_{p}^{n}})=\{f\in\mathscr{H}(\mathbb{B}_{\ell_{p}^{n}}):
\sup_{z\in\mathbb{B}_{\ell_{p}^{n}}}|f(z)|\leq1\}$. We use
$\mathcal{R}(\mathbb{B}_{\ell_{p}^{n}})$ to denote the largest
non-negative number $\rho$ with the property that if
$f(z)=\sum_{\alpha}a_{\alpha}z^{\alpha}\in\mathscr{H}_{1}(\mathbb{B}_{\ell_{p}^{n}})$,
then $\sum_{\alpha}|a_{\alpha}z^{\alpha}|\leq 1$ in
$\rho\mathbb{B}_{\ell_{p}^{n}}:=\{z\in\mathbb{C}^{n}:~\|z\|_{p}<\rho\}$. We also call
$\mathcal{R}(\mathbb{B}_{\ell_{p}^{n}})$ the $n$-dimensional Bohr
radius.

When $n>1$,
the exact value of the Bohr radius
$\mathcal{R}(\mathbb{B}_{\ell_{p}^{n}})$
is still unknown.
In the following result, Boas and  Khavinson \cite[Theorem 2]{BK} showed the upper estimate and Defant et al. in
\cite[Theorem 2]{D-2011} showed the lower estimate:

\begin{Thm}\label{eq-Def-2011}
There exists a constant
$\mathbf{b}_{n}$ such that
%
\[
\mathcal{R}(\mathbb{B}_{\ell_{\infty}^{n}})=\mathbf{b}_{n}\sqrt{\frac{\log
n}{n}}
\]
 with
$1/\sqrt{2}+o(1)\leq \mathbf{b}_{n}\leq2.$
\end{Thm}
Bayart, Pellegrino and Seoane-Sep\'ulveda  (\cite{BPS})  proved the
exact asymptotical behaviour of
$\mathcal{R}(\mathbb{B}_{\ell_{\infty}^{n}})$ as follows:

\begin{Thm}\label{IB}
\[
\lim_{n\rightarrow\infty}\frac{\mathcal{R}(\mathbb{B}_{\ell_{\infty}^{n}})}{\sqrt{\frac{\log
n}{n}}}=1.
\]
\end{Thm}
In the case  $p\in[1,\infty)$, the results of Boas, Defant and  Frerick
 from \cite{B-2000,DF-11} showed the following result:

\begin{Thm}\label{Bohr-t1}
There is a constant
$C\geq1$ such that for each $p\in[1,\infty)$ and  $n\geq2$

\[
\frac{1}{C}\left(\frac{\log
n}{n}\right)^{1-\frac{1}{\min\{p,2\}}}\leq\mathcal{R}(\mathbb{B}_{\ell_{p}^{n}})\leq
  C\left(\frac{\log n}{n}\right)^{1-\frac{1}{\min\{p,2\}}}.
\]
\end{Thm}


The first aim of this paper is to establish some Schwarz-Pick type lemmas and discuss their applications.
Based on Theorem \Ref{Liu-2019} by Liu, we give an answer to Question \ref{q-1}
(see Theorem \ref{thm-0}). In Theorem \ref{thm-2+}, we use a new proof method to
improve and generalize Theorem \Ref{Z-1} into the  sharp form. In particular, in Theorem \ref{thm-1},
we also give a sharp estimate for pluriharmonic functions with values in $\mathbf{B}_{\ell_{p}^{\nu}}$. Moreover,
we  obtain the Schwarz-Pick type estimate for pluriharmonic functions which generalizes and improves Theorem \Ref{Z-2} into the sharp form
on $\mathbb{B}_{\ell_{\infty}^{n}}$
(see Theorem \ref{thm-3+}) and we will use it to discuss the Lipschitz characteristic of harmonic functions on a domain in ${\mathbb C}$
(see Proposition \ref{prop-1+}).

The second purpose of this paper is to investigate the coefficient type Schwarz-Pick  lemmas and give their applications.
In Theorems \ref{lem-3.1} and \ref{lem-0},  we extend Theorem \Ref{Coe-2} to the $n$-dimensional case.
Then, by using different proof techniques, we will use  Theorem \ref{lem-3.1} to extend (\ref{Col-89})
and Theorem \Ref{Higher-1} to pluriharmonic functions on
the unit ball $\mathbb{B}_{\ell_{2}^{n}}$, and give an estimate for
the partial derivatives of arbitrary order
(see Theorem \ref{thm-4}).
Furthermore, by using Theorem \ref{lem-0}, we will establish a sharp Schwarz-Pick type inequality of arbitrary order
for some class of pluriharmonic functions on $\mathbb{B}_{\ell_{\infty}^{n}}$
(see Theorem \ref{thm-2}). At last,
by applying Theorem \ref{lem-3.1}, we obtain a Bohr type inequality which is an analogue of Theorems \Ref{eq-Def-2011}
and \Ref{Bohr-t1}
(see Theorem \ref{thm-2.1x}). We also apply
 Theorem \ref{lem-0} to show that
Theorems \Ref{Bohr-t3}, \Ref{IB} and \Ref{Bohr-t1}  also
hold for a more general class
of pluriharmonic functions on $\mathbb{B}_{\ell_{p}^{n}}$
(see Theorem  \ref{thm-2.0x}).


The paper is organized as follows.
In section \ref{csw-sec1},
we give the statements of our results.
In sections \ref{csw-sec2} and \ref{csw-sec3},
we give the proofs of our main results.

\section{Preliminaries and  main results }\label{csw-sec1}


\subsection{The Schwarz-Pick type lemmas and their applications}

Based on Theorem \Ref{Liu-2019} by Liu, we give an answer to Question \ref{q-1} as follows.

\begin{thm}\label{thm-0}
For $p\in(1,\infty)$ and $n\geq3$, let $u$ be a harmonic function of $\mathbf{B}_{\ell_{2}^{n}}$ into $\mathbf{B}_{\ell_{p}^{\nu}}$,
where $\nu$ is a positive integer.
 Then, for $x\in\mathbf{B}_{\ell_{2}^{n}}$, we have the following sharp inequality:
\be\label{t-0} \big|\nabla \|u(x)\|_{p}\big|\leq\frac{c_{n}}{1-\|x\|_{2}^{2}}
 \left\{\int_{-1}^{1}\frac{\left|t-\frac{n-2}{n}\|x\|_{2}\right|(1-t^{2})^{\frac{n-3}{2}}}{(1-2t\|x\|_{2}+\|x\|_{2}^{2})^{\frac{n-2}{2}}}dt\right\},\ee
where $c_{n}$ is the same as in Theorem \Ref{Liu-2019}.
\end{thm}

A twice continuously differentiable complex-valued function $f$
defined on a domain $\Omega\subset\mathbb{C}^{n}$ is called a {\it
pluriharmonic function} if for each fixed $z\in \Omega$ and
$\theta\in\partial\mathbb{B}_{\ell_{2}^{n}}$, the function
$f(z+\zeta\theta)$ is harmonic in $\{\zeta:\; |\zeta|<
d_{\Omega}(z)\}$ (cf.
\cite{DHK-2011,R2,Vl}). Obviously, all pluriharmonic functions are harmonic.
If
$\Omega\subset\mathbb{C}^{n}$ is a simply connected domain
containing the origin, then a function $f:\,\Omega\rightarrow
\mathbb{C}$ is pluriharmonic if and only if $f$ has a representation
$f=h+\overline{g},$ where $h$ and $g$ are holomorphic in $\Omega$
with $g(0)=0$ (see \cite{Vl}). Furthermore, a twice continuously
differentiable real-valued function in a simply connected domain
$\Omega$ is pluriharmonic if and only if it is the real part of some
holomorphic function on $\Omega$. In particular, if $n=1$, then
pluriharmonic functions are planar harmonic functions (cf. \cite{Du}).


In the following, we  improve and generalize Theorem \Ref{Z-1} into the  sharp form.

\begin{thm}\label{thm-2+}
For $n\geq1$ and $p\in(1,\infty)$, let $f$ be a pluriharmonic function of $\mathbb{B}_{\ell_{2}^{n}}$ into $\mathbb{B}_{\ell_{p}^{\nu}}$, where  $\nu$ is a positive integer.
Then the following  inequality holds:
\be\label{eq-ch-3.3g}|\nabla \|f(z)\|_{p}|\leq\frac{4}{\pi}\frac{1}{1-\|z\|_{2}^{2}},~z\in\mathbb{B}_{\ell_{2}^{n}}.\ee
Furthermore,  {\rm (\ref{eq-ch-3.3g})} is sharp for each $z\in \mathbb{B}_{\ell_{2}^{n}}$.
\end{thm}

In particular, if $u$ is a pluriharmonic function of $\mathbb{B}_{\ell_{2}^{n}}$ into $\mathbf{B}_{\ell_{p}^{\nu}}$, then we have
a better estimate as follows, where $\nu$ is a positive integer.

\begin{thm}\label{thm-1}
For $n\geq1$ and $p\in(1,\infty)$, let $u$ be a pluriharmonic function of $\mathbb{B}_{\ell_{2}^{n}}$ into $\mathbf{B}_{\ell_{p}^{\nu}}$,
 where  $\nu$ is a positive integer. Then for $z\in\mathbb{B}_{\ell_{2}^{n}}$,
\be\label{eq-th-1}\big|\nabla \|u(z)\|_{p}\big|\leq\frac{4}{\pi}\frac{1-\|u(z)\|_{p}^{2}}{1-\|z\|^{2}_{2}}.\ee
The inequality {\rm(\ref{eq-th-1})} is sharp for each $z$.
\end{thm}

We give the following Schwarz-Pick type estimate for pluriharmonic functions which generalizes and improves Theorem \Ref{Z-2} into the sharp form on $\mathbb{B}_{\ell_{\infty}^{n}}$ and we will use it to discuss the Lipschitz characteristic of harmonic functions on a domain in ${\mathbb C}$.

\begin{thm}\label{thm-3+}
For $n\geq1$, let
$f=(f_{1},\ldots,f_{\nu}):~\mathbb{B}_{\ell_{\infty}^{n}}\mapsto\mathbb{B}_{\ell_{2}^{\nu}}$ be a pluriharmonic
function, where   $\nu$ is a positive integer.  Then, for $z\in\mathbb{B}_{\ell_{\infty}^{n}}$, we have
\be\label{qq-ch1.1g}\sum_{j=1}^{\nu}\sum_{k=1}^{n}\left(\left|\frac{\partial f_{j}(z)}{\partial z_{k}}\right|^{2}+
\left|\frac{\partial f_{j}(z)}{\partial\overline{z}_{k}}\right|^{2}\right)(1-|z_{k}|^{2})^{2}
\leq1-\|f(z)\|_{2}^{2}.\ee
and
\be\label{qq-ch1.01g}\sum_{j=1}^{\nu}\sum_{k=1}^{n}\left(\left|\frac{\partial f_{j}(z)}{\partial z_{k}}\right|^{2}+
\left|\frac{\partial f_{j}(z)}{\partial\overline{z}_{k}}\right|^{2}\right)
\leq\frac{1-\|f(z)\|_{2}^{2}}{(1-\|z\|_{\infty}^{2})^{2}}.\ee
Moreover, the inequality {\rm (\ref{qq-ch1.1g})} is sharp for each $z\in \mathbb{B}_{\ell_{\infty}^{n}}$.
\end{thm}

For a given constant $\alpha\in(0,\infty)$ and a given subset $\Omega$ of $\mathbb{C}$, a function $f:~\Omega\mapsto\mathbb{C}^{n}$ is said to belong
to the Lipschitz space $\Lambda_{\alpha}(\Omega)$ if there is a constant $C\geq1$ such that for all $z,w\in\Omega$,
$$\|f(z)-f(w)\|_{2}\leq C|z-w|^{\alpha}.$$
For the related investigation of the Lipschitz space of analytic functions, we refer to \cite{Dy1,Dy-2004,P-1999}.

A domain $\Omega\subseteq\mathbb{C}$ is said to be linearly
connected if there exists a constant $M>0$ such that any two points
$v_1,\, v_2\in\Omega$ can be connected by a smooth curve
$\gamma\subset\Omega$ with length $\ell(\gamma)\leq
M| v_1-v_2|$ (see \cite{CHHK-2014}).
By applying Theorem \ref{thm-3+}, we get the following result.

\begin{prop}\label{prop-1+}
Suppose that $\Omega$ is a linearly connected proper subdomain of $\mathbb{C}$.  Let
$f:~\Omega\mapsto\mathbb{C}^{n}$ be a harmonic
function, where $n\geq2$. If $\|f\|_{2}^{2}\in\Lambda_{2}(\Omega)$, then $f\in\Lambda_{1}(\Omega)$.
\end{prop}

\subsection{The coefficient type Schwarz-Pick  lemmas and their applications}

In the following, we extend Theorem \Ref{Coe-2} to the $n$-dimensional case.

\begin{thm}\label{lem-3.1}
For $p\in[1,\infty]$, let
$f(z)=\sum_{\alpha}a_{\alpha}z^{\alpha}+\sum_{\alpha}\overline{b}_{\alpha}\overline{z}^{\alpha}$ be a pluriharmonic function of $\mathbb{B}_{\ell_{p}^{n}}$
into $\mathbb{D}$.
 Then for
all $|\alpha|\geq1$,

\begin{equation}\label{c-1.7s} |a_{\alpha}|+|b_{\alpha}|\leq
\frac{4}{\pi}\left(\frac{|\alpha|^{|\alpha|}}{\alpha^{\alpha}}\right)^{\frac{1}{p}}.
\end{equation}
Moreover, the constant $4/\pi$ in {\rm (\ref{c-1.7s})} is sharp.

\end{thm}

For $p\in[1,\infty]$, let $\mathscr{PH}(\mathbb{B}_{\ell_{p}^{n}})$
denote the set of all pluriharmonic functions from
$\mathbb{B}_{\ell_{p}^{n}}$ into $\mathbb{C}$.  By analogy with Theorem \Ref{Coe-2}, we establish a coefficient type Schwarz-Pick lemma for
a general case as follows.

\begin{thm}\label{lem-0}
 If
$f(z)=\sum_{\alpha}a_{\alpha}z^{\alpha}+\sum_{\alpha}\overline{b}_{\alpha}\overline{z}^{\alpha}\in\mathscr{PH}(\mathbb{B}_{\ell_{p}^{n}})$
with $\sup_{z\in\mathbb{B}_{\ell_{p}^{n}}}{\rm Re}(f(z))\leq 1$,
then for all $k\in\{1,2,\ldots\}$,

\begin{equation}\label{c-1.01}
\sup_{z\in\mathbb{B}_{\ell_{p}^{n}}}\left|\sum_{|\alpha|=k}(a_{\alpha}+b_{\alpha})z^{\alpha}\right|\leq2\big(1-{\rm
Re}(f(0))\big),
\end{equation} and for all $|\alpha|\geq1$,

\begin{equation}\label{c-1.7} |a_{\alpha}+b_{\alpha}|\leq
2\left(\frac{|\alpha|^{|\alpha|}}{\alpha^{\alpha}}\right)^{\frac{1}{p}}\big(1-{\rm
Re}(f(0))\big).
\end{equation}
Moreover, the constant $2$ in {\rm (\ref{c-1.01})} and {\rm
(\ref{c-1.7})} are sharp.
\end{thm}


In the following, we will use  Theorem \ref{lem-3.1} to extend (\ref{Col-89})
and Theorem \Ref{Higher-1} to pluriharmonic functions on
the unit ball $\mathbb{B}_{\ell_{2}^{n}}$, and give an estimate for
the partial derivatives of arbitrary order. One should note that the higher dimensional case is very
different from the one dimensional situation and, because we are
dealing with partial derivatives of arbitrary order, the method of
proof for (\ref{Col-89}) from
 \cite{Co-1989} can not be used. The result is as follows.


\begin{thm}\label{thm-4}
Suppose that  $f$ is a pluriharmonic function of  $\mathbb{B}_{\ell_{2}^{n}}$ into $\mathbb{D}$, where $n\geq 1$.
 \begin{enumerate}
\item[(i)] If $n=1$, then for
$z\in\mathbb{B}_{\ell_{2}^{1}}=\mathbb{D}$,
$$\left|\frac{\partial^{m} f(z)}{\partial z^{m}}\right|+\left|\frac{\partial^{m} f(z)}{\partial
\overline{z}^{m}}\right|
\leq\frac{4}{\pi}m!\frac{(1+|z|)^{m-1}}{(1-|z|^{2})^{m}},$$ where
$m\geq1$. The constant $4/\pi$ in this inequality can not be improved.

\item[(ii)] If $n\geq 2$, then for $z\in\mathbb{B}_{\ell_{2}^{n}}$,
$$\left|\frac{\partial^{|m|} f(z)}{\partial z_{1}^{m_{1}}\cdots
\partial z_{n}^{m_{n}}}\right|+\left|\frac{\partial^{|m|} f(z)}{\partial
\overline{z}_{1}^{m_{1}}\cdots
\partial \overline{z}_{n}^{m_{n}}}\right|
\leq\frac{4}{\pi}n^{\frac{|m|}{2}}{n+|m|-1\choose
n-1}|m|!\frac{\prod_{j=1}^{n}(1+|z_j|)^{m_j}}{(1-\|z\|^{2}_{2})^{|m|}},$$
where $m=(m_{1},\ldots,m_{n})\neq0$ is a multi-index.
\end{enumerate}
\end{thm}

\begin{rem}\label{rem-2y}
 We remark that
if $n=m=1$, then Theorem \ref{thm-4} coincides with (\ref{Col-89}) (or \cite[Theorem 3]{Co-1989} ).
Moreover, the growth rate of the formula ${n+|m|-1\choose
n-1}$ in Theorem \ref{thm-4}  can be estimated. It follows from
Stirling's formula (cf. \cite{N-D}) that there is an absolute
constant $c\in[1,e]$ such that for any $n$ and $m$,

\begin{equation}\label{Num-1} {n+|m|-1\choose n-1}\leq
c^{|m|}\left(1+\frac{n}{|m|}\right)^{|m|}.
\end{equation}
In particular, for the extension of (\ref{Col-89}) on the polydisc
$\mathbb{B}_{\ell_{\infty}^{n}}$, $n\geq 2$, see \cite{CR}. It is well known
that there are no biholomorphic mappings between
$\mathbb{B}_{\ell_{\infty}^{n}}$ and $\mathbb{B}_{\ell_{2}^{n}}$
in the case $n\geq 2$.
Hence,  the research methods to deal with these two situations are
completely different (see \cite{R1,R2}).
\end{rem}

By using Theorem \ref{lem-0}, we will establish a sharp Schwarz-Pick type inequality of arbitrary order
for $f\in\mathscr{PH}(\mathbb{B}_{\ell_{\infty}^{n}})$ with $\sup_{z\in\mathbb{B}_{\ell_{\infty}^{n}}}{\rm Re}(f(z))\leq 1$.

\begin{thm}\label{thm-2}
Let $f\in\mathscr{PH}(\mathbb{B}_{\ell_{\infty}^{n}})$  with
$\sup_{z\in\mathbb{B}_{\ell_{\infty}^{n}}}{\rm Re}(f(z))\leq 1$.
Then

\begin{equation}\label{c-1.5}\left|\frac{\partial^{|m|} f(z)}{\partial
z_{1}^{m_{1}}\cdots
\partial z_{n}^{m_{n}}}+ \overline{\frac{\partial^{|m|} f(z)}{\partial
\overline{z}_{1}^{m_{1}}\cdots
\partial \overline{z}_{n}^{m_{n}}}}\right|\leq2(m!)(1-{\rm Re}(f(z)))\frac{(1+\|z\|_{\infty})^{|m|-N}}{(1-\|z\|_{\infty}^{2})^{|m|}},
 \end{equation}
where $m=(m_{1},\ldots,m_{n})\neq0$ is a multi-index and $N$ is the
number of the indices $j$ such that $m_j\neq0$. Furthermore, the
constant $2$ in {\rm (\ref{c-1.5})} cannot be improved.
\end{thm}

We will use Theorems \ref{lem-3.1} and \ref{lem-0} to investigate
the Bohr phenomenon of    complex-valued
pluriharmonic functions.

Let $\mathcal{R}^{\ast}_{P}(\mathbb{B}_{\ell_{p}^{n}})$ denote the
$n$-dimensional Bohr radius: the largest number $\rho$ such that if
$f(z)=\sum_{\alpha}a_{\alpha}z^{\alpha}+\sum_{\alpha}\overline{b}_{\alpha}\overline{z}^{\alpha}\in\mathscr{PH}(\mathbb{B}_{\ell_{p}^{n}})$
with $b_0=0$ and $\sup_{z\in\mathbb{B}_{\ell_{p}^{n}}}|f(z)|\leq 1$,
then

\begin{equation}\label{hhk-1}
\sum_{k=1}^{\infty}\sum_{|\alpha|=k}(|a_{\alpha}|+|b_{\alpha}|)|z^{\alpha}|\leq 1
\end{equation} when $z\in\rho\mathbb{B}_{\ell_{p}^{n}}.$


By applying Theorem \ref{lem-3.1}, we obtain a Bohr type inequality which is an analogue of Theorems \Ref{eq-Def-2011}
and \Ref{Bohr-t1}
as follows.
Note that
the proof of
Theorem \Ref{eq-Def-2011}  depends on  Wiener's result (see
\cite{D-2011}). However, the proof of Theorem \ref{thm-2.1x}
does not need  Wiener's result.

\begin{thm}\label{thm-2.1x}
Let $p\in[1,\infty]$ and $n\geq2$. The $n$-dimensional Bohr radius
$\mathcal{R}^{\ast}_{P}(\mathbb{B}_{\ell_{p}^{n}})$ satisfies
\begin{equation}\label{eq-T-1}
C_{1}\left(\frac{1}{n}\right)^{1-\frac{1}{\min\{p,2\}}}\leq\mathcal{R}^{\ast}_{P}(\mathbb{B}_{\ell_{p}^{n}})\leq
C_{2}\left(\frac{\log n}{n}\right)^{1-\frac{1}{\min\{p,2\}}},
\end{equation}
where $C_{j}>0~(j=1,2)$ are absolute constants.
\end{thm}

\begin{rem}
Is there the Bohr's phenomenon if one replace (\ref{hhk-1}) by
$\sum_{\alpha}(|a_{\alpha}|+|b_{\alpha}|)|z^{\alpha}|\leq 1$? The
following example shows that the answer is no.
\end{rem}

\begin{exam}
For $z\in\mathbb{B}_{\ell_{p}^{n}}$,
let
$\mathcal{F}_{k}(z)=\mbox{Re}(z_1)\sin\frac{1}{k}+i\cos\frac{1}{k}$,
where $k\in\{1,2,\ldots\}$. Then
$\mathcal{F}_{k}\in\mathscr{PH}(\mathbb{B}_{\ell_{p}^{n}})$ and
$\sup_{z\in\mathbb{B}_{\ell_{p}^{n}}}|\mathcal{F}_{k}(z)|\leq 1$ for
all $k\in\{1,2,\ldots\}$. Suppose that there is $\rho_{0}>0$ such
that
$$\left|\sin\frac{1}{k}\right||z_{1}|+\left|\cos\frac{1}{k}\right|\leq 1$$
for $\|z\|_{p}<\rho_{0}$ and all $k\in\{1,2,\ldots\}$. This
contradicts $\lim_{k\rightarrow\infty}\cos\frac{1}{k}=1$.
\end{exam}

For $p\in[1,\infty]$,
let
$$\mathscr{PH}_{+}(\mathbb{B}_{\ell_{p}^{n}})=\{f\in\mathscr{PH}(\mathbb{B}_{\ell_{p}^{n}}):~\sup_{z\in\mathbb{B}_{\ell_{p}^{n}}}{\rm
Re}(f(z))\leq 1~\mbox{and}~f(0)\geq0\}.$$  We use
$\mathcal{R}_{P}(\mathbb{B}_{\ell_{p}^{n}})$ to denote the
$n$-dimensional Bohr radius: the largest number $\rho$ such that if
$f(z)=\sum_{\alpha}a_{\alpha}z^{\alpha}
+\sum_{\alpha}\overline{b}_{\alpha}\overline{z}^{\alpha}
\in\mathscr{PH}_{+}(\mathbb{B}_{\ell_{p}^{n}})$ with $b_0=0$, then
$\sum_{\alpha}|(a_{\alpha}+b_{\alpha})z^{\alpha}|\leq 1$ when
$z\in\rho\mathbb{B}_{\ell_{p}^{n}}.$
In fact, the assumption $f(0)\geq0$ in
$\mathscr{PH}_{+}(\mathbb{B}_{\ell_{p}^{n}})$ is necessary. Without
this condition, there is no Bohr's phenomenon on the set
$$\mathscr{PH}_{1}(\mathbb{B}_{\ell_{p}^{n}})=\{f\in\mathscr{PH}(\mathbb{B}_{\ell_{p}^{n}}):~\sup_{z\in\mathbb{B}_{\ell_{p}^{n}}}{\rm
Re}(f(z))\leq 1\}.$$ Here is an example.

\begin{exam}
For $z=(z_{1},\ldots,z_{n})\in\mathbb{B}_{\ell_{p}^{n}}$, let
$f(z)=\mbox{Re}\left(\frac{-2z_{1}}{1-z_{1}}\right)=-\sum_{j=1}^{\infty}z_{1}^{j}-\sum_{j=1}^{\infty}\overline{z}_{1}^{j}$,
and let
$F_{k}(z)=\left(\sin\frac{1}{k}\right)f(z)+i\cos\frac{1}{k}$, where
$k\in\{1,2,\ldots\}$. Then
$F_{k}\in\mathscr{PH}_{1}(\mathbb{B}_{\ell_{p}^{n}})$. Suppose that
there is $\rho_{0}>0$ such that
$$\sum_{j=1}^{\infty}2\left|\sin\frac{1}{k}\right||z_{1}|^{j}+\left|\cos\frac{1}{k}\right|\leq 1$$
for $\|z\|_{p}<\rho_{0}$ and all $k\in\{1,2,\ldots\}$. This
contradicts $\lim_{k\rightarrow\infty}\cos\frac{1}{k}=1$. Therefore,
there is no Bohr's phenomenon on
$\mathscr{PH}_{1}(\mathbb{B}_{\ell_{p}^{n}}).$
\end{exam}

 In the following, by using Theorem \ref{lem-0}, we shall show that we can go
further: Theorems \Ref{Bohr-t3}, \Ref{IB} and \Ref{Bohr-t1}  also
hold for a more general class
$\mathscr{PH}_{+}(\mathbb{B}_{\ell_{p}^{n}})$.

\begin{thm}\label{thm-2.0x}
Let $p\in[1,\infty]$. Then the $n$-dimensional Bohr radius
$\mathcal{R}_{P}(\mathbb{B}_{\ell_{p}^{n}})$ satisfies

$$ \begin{cases}
\displaystyle \mathcal{R}_{P}(\mathbb{B}_{\ell_{p}^{n}})=\frac{1}{3}, &\, n=1~\mbox{and}~p\in[1,\infty],\\
\displaystyle\\ \displaystyle \frac{1}{C}\left(\frac{\log
n}{n}\right)^{1-\frac{1}{\min\{p,2\}}}\leq\mathcal{R}_{P}(\mathbb{B}_{\ell_{p}^{n}})\leq
  C\left(\frac{\log n}{n}\right)^{1-\frac{1}{\min\{p,2\}}}, &\,
  n\geq2~\mbox{and}~p\in[1,\infty),\\
\displaystyle\\ \displaystyle
\lim_{n\rightarrow\infty}\frac{\mathcal{R}_{P}(\mathbb{B}_{\ell_{p}^{n}})}{\sqrt{\frac{\log
n}{n}}}=1,&\, n\geq2~\mbox{and}~p=\infty,
\end{cases}$$ where $C\geq1$ is an absolute constant.
In particular, if $n=1$, then the constant $1/3$ is sharp.
\end{thm}

\begin{rem}\label{rem-1}
It is obvious that if $f\in
\mathscr{H}_{1}(\mathbb{B}_{\ell_{\infty}^{n}})$ with $f(0)=0$, then
$f\in\mathscr{PH}_{+}(\mathbb{B}_{\ell_{\infty}^{n}})$. Furthermore,
if $f(z)=\sum_{\alpha}a_{\alpha}z^{\alpha}\in
\mathscr{H}_{1}(\mathbb{B}_{\ell_{\infty}^{n}})$ with $f(0)\neq0$,
then $e^{-i\arg
f(0)}f\in\mathscr{PH}_{+}(\mathbb{B}_{\ell_{\infty}^{n}})$. Note
that $\sum_{\alpha}|e^{-i\arg
f(0)}a_{\alpha}z^{\alpha}|=\sum_{\alpha}|a_{\alpha}z^{\alpha}|$.
 Hence Theorem \ref{thm-2.0x}
is an improvement of
Theorems  \Ref{IB} and \Ref{Bohr-t1}.
In particular, the proof of
Theorem \Ref{IB} depends on  Wiener's result (see
\cite{BPS}). However, the proof of Theorem \ref{thm-2.0x}
does not need  Wiener's result.
\end{rem}

\section{The Schwarz-Pick type lemmas and their applications}\label{csw-sec2}

For a differentiable mapping $f:\mathbf{B}_{\ell_{2}^{n}}\to \mathbb{R}^k$,
let $Df(z)$ denote the Fr\'{e}chet derivative of $f$ at $z$, where  $k$ is a positive integer.

\begin{lem}
\label{lem-Schwarz-def}
For $p\in(1,\infty)$, let $f$ be a $C^1$ class  function of $\mathbf{B}_{\ell_{2}^{n}}$
into $\mathbf{B}_{\ell_{p}^{k}}$, where  $k$ is a positive integer.
Then,
\be\label{lem-ch1}
|\nabla \|f(z)\|_{p}|=\sup_{\| \theta\|_2=1}\limsup_{\rho\rightarrow 0^{+}}
\frac{|\|f(z+\rho \theta)\|_{p}-\|f(z)\|_{p}|}{\rho},
\quad
z\in \mathbf{B}_{\ell_{2}^{n}}
\ee
holds.
\end{lem}

\begin{proof} We divide the proof into two cases.
\bca\label{cla-l-1}
Let $z\in\mathscr{P}:=\{\varsigma\in\mathbf{B}_{\ell_{2}^{n}}:~f(\varsigma)=0\}$  be fixed. \eca
Elementary calculations lead to
\beqq
\sup_{\| \theta\|_2=1}\limsup_{\rho\rightarrow 0^{+}}
\frac{|\|f(z+\rho \theta)\|_{p}-\|f(z)\|_{p}|}{\rho}&=&\sup_{\| \theta\|_2=1}\limsup_{\rho\rightarrow 0^{+}}
\frac{\|f(z+\rho \theta)-f(z)\|_{p}}{\rho}\\
&=&\sup_{\| \theta\|_2=1}\| Df(z)\theta\|_p
\eeqq
and
\beqq
|\nabla \|f(z)\|_{p}|
&=&
\limsup_{w\rightarrow z}\frac{\|f(w)-f(z)\|_{p}}{\|w-z\|_{2}}
=
\limsup_{w\rightarrow z}\left\| Df(z)\frac{w-z}{\|w-z\|_{2}}\right\|_p
\\
&=&
\sup_{\| \theta\|_2=1}\| Df(z)\theta\|_p,
\eeqq
which imply (\ref{lem-ch1}).

\bca\label{cla-l-2}
Let $z\in\mathbf{B}_{\ell_{2}^{n}}\setminus\mathscr{P}$  be fixed. \eca
In this case,
let $\Psi(w)=\| f(w)\|_p$ for $w\in \mathbf{B}_{\ell_{2}^{n}}$. It is not difficult to know that
$\Psi$ is $C^1$ on a neighbourhood of $z$. Then (\ref{lem-ch1}) follows from the following
two formulas:

\beqq
|\nabla \|f(z)\|_{p}|
&=&
\limsup_{w\rightarrow z}\frac{|\Psi(w)-\Psi(z)|}{\|w-z\|_{2}}
=
\limsup_{w\rightarrow z}\left| D\Psi(z)\frac{w-z}{\|w-z\|_{2}}\right|
\\
&=&
\sup_{\| \theta\|_2=1}|D\Psi(z)\theta|
\eeqq
and
\beqq
\sup_{\| \theta\|_2=1}\limsup_{\rho\rightarrow 0^{+}}
\frac{|\|f(z+\rho \theta)\|_{p}-\|f(z)\|_{p}|}{\rho}&=&\sup_{\| \theta\|_2=1}\limsup_{\rho\rightarrow 0^{+}}
\frac{|\Psi(z+\rho\theta)-\Psi(z)|}{\rho}\\
&=&\sup_{\| \theta\|_2=1}|D\Psi(z)\theta|.
\eeqq
This completes the proof.
\end{proof}

\subsection*{ The proof of Theorem \ref{thm-0}}
For $p\in(1,\infty)$ and $n\geq3$, let $u=(u_{1},\ldots,u_{\nu})$ be a harmonic function of $\mathbf{B}_{\ell_{2}^{n}}$ into $\mathbf{B}_{\ell_{p}^{\nu}}$.
 We divide the proof of (\ref{t-0}) into two steps.
 \bst\label{bst-0.1} We first estimate $\big|\nabla \|u(x)\|_{p}\big|$ for $x\in\Omega:=\{y\in\mathbf{B}_{\ell_{2}^{n}}:~u(y)\neq0\}$.
 \est
By Lemma \ref{lem-Schwarz-def}, we have
\be\label{eq-1.4g}
|\nabla \|u(x)\|_{p}|=\max_{\vartheta\in\partial\mathbf{B}_{\ell_{2}^{n}}}\left|\sum_{j=1}^{n}\frac{\partial \|u(x)\|_{p}}{\partial x_{j}}\vartheta_{j}\right|
=
\left(\sum_{j=1}^{n}\left|\frac{\partial \|u(x)\|_{p}}{\partial x_{j}}\right|^{2}\right)^{\frac{1}{2}}.
\ee
Elementary calculations yield
\be\label{eq-1.5g}\frac{\partial \|u(x)\|_{p}}{\partial x_{j}}=\sum_{k=1}^{\nu}\frac{|u_{k}(x)|^{p-2}u_{k}(x)}{\|u(x)\|_{p}^{p-1}}\frac{\partial u_{k}(x)}{\partial x_{j}}
=\left\langle\frac{\partial u(x)}{\partial x_{j}},\eta(x)\right\rangle\ee
 for $j\in\{1,\ldots,n\}$, where
$$\frac{\partial u(x)}{\partial x_{j}}=\left(\frac{\partial u_{1}(x)}{\partial x_{j}},\ldots,\frac{\partial u_{\nu}(x)}{\partial x_{j}}\right)$$ and
 $$\eta(x)=\left(\frac{|u_{1}(x)|^{p-2}u_{1}(x)}{\|u(x)\|_{p}^{p-1}},\ldots,\frac{|u_{\nu}(x)|^{p-2}u_{\nu}(x)}{\|u(x)\|_{p}^{p-1}}\right).$$
It is not difficult to know that $\eta\in\partial\mathbf{B}_{\ell_{q}^{\nu}}$, where $\frac{1}{q}+\frac{1}{p}=1.$
By (\ref{eq-1.4g}) and (\ref{eq-1.5g}), we obtain

\be\label{eq-r1}|\nabla \|u(x)\|_{p}|=\left\|(Du(x))^{T}(\eta(x))^{T}\right\|_{2},
\ee where $``T"$ denotes the transpose of a matrix and
$$Du(x)=\left(\begin{array}{cccc}
\ds \frac{\partial u_{1}(x)}{\partial x_{1}}\;\cdots\;
 \frac{\partial u_{1}(x)}{\partial x_{n}}\\[4mm]
\vdots\;\; \;\;\;\;\;\cdots\;\;\;\;\; \;\;\vdots \\[2mm]
 \ds \frac{\partial u_{\nu}(x)}{\partial x_{1}}\; \cdots\;
 \frac{\partial u_{\nu}(x)}{\partial x_{n}}
\end{array}\right)
$$ is the Fr\'echet derivative of $u$ at $x$.

For any fixed $\alpha=(\alpha_{1},\ldots,\alpha_{\nu})\in\partial\mathbf{B}_{\ell_{q}^{\nu}}$, let
\be\label{eq-1c} u_{\alpha}(y)=\langle u(y),\alpha\rangle,
~y\in\mathbf{B}_{\ell_{2}^{n}}.\ee Then the H\"older inequality implies
$$|u_{\alpha}|=|\langle u,\alpha\rangle|\leq\left(\sum_{j=1}^{\nu}|u_{j}|^{p}\right)^{\frac{1}{p}}
\left(\sum_{j=1}^{\nu}|\alpha_{j}|^{q}\right)^{\frac{1}{q}}<1.$$
An application of Theorem \Ref{Liu-2019} to the bounded harmonic function $u_{\alpha}$ gives that, for all $x\in\mathbf{B}_{\ell_{2}^{n}}$,

\beqq
\|\nabla u_{\alpha}(x)\|_{2}&=&\left\|(Du(x))^{T}\alpha^{T}\right\|_{2}\\
&\leq&\frac{c_{n}}{1-\|x\|_{2}^{2}}
 \left\{\int_{-1}^{1}\frac{\left|t-\frac{n-2}{n}\|x\|_{2}\right|(1-t^{2})^{\frac{n-3}{2}}}{(1-2t\|x\|_{2}+\|x\|_{2}^{2})^{\frac{n-2}{2}}}dt\right\},
 \eeqq
which, together with (\ref{eq-r1}), implies that (\ref{t-0}) holds.

 \bst\label{bst-0.2} Next, we estimate $\big|\nabla \|u(x)\|_{p}\big|$ for $x\in\mathbf{B}_{\ell_{2}^{n}}\setminus\Omega$.
 \est

By Lemma \ref{lem-Schwarz-def}, we have
\beq\label{eq-1.1f}
|\nabla \|u(x)\|_{p}|&=&\limsup_{y\rightarrow x}\frac{|\|u(x)\|_{p}-\|u(y)\|_{p}|}{\|x-y\|_{2}}\\ \nonumber
&=&\max_{\vartheta\in\partial\mathbf{B}_{\ell_{2}^{n}}}\limsup_{\rho\rightarrow 0^{+}}\frac{\|u(x+\rho\vartheta)-u(x)\|_{p}}{\rho}\\ \nonumber
&=&\max_{\vartheta\in\partial\mathbf{B}_{\ell_{2}^{n}}}\left( \sum_{k=1}^{\nu}\left|\langle\nabla u_{k}(x),\vartheta\rangle\right|^{p}
\right)^{\frac{1}{p}}.
\eeq
For any fixed $\alpha=(\alpha_{1},\ldots,\alpha_{\nu})\in\partial\mathbf{B}_{\ell_{q}^{\nu}}$,
where $\frac{1}{q}+\frac{1}{p}=1$,
let $u_{\alpha}$
be the harmonic function defined in (\ref{eq-1c}).
It follows from
Theorem \Ref{Liu-2019} that

\beq\label{eq-ch-1-1}
\max_{\vartheta\in\partial\mathbf{B}_{\ell_{2}^{n}}}\left|\sum_{j=1}^{\nu}\langle\nabla u_{j}(x),\vartheta\rangle\alpha_{j}\right|&=&
\max_{\vartheta\in\partial\mathbf{B}_{\ell_{2}^{n}}}\lim_{\rho\to0^{+}}\frac{\left|u_{\alpha}(x+\rho\vartheta)-u_{\alpha}(x)\right|}{\rho}\\ \nonumber
&=&\|\nabla u_{\alpha}(x)\|_{2}\\ \nonumber
&\leq&\frac{c_{n}}{1-\|x\|_{2}^{2}}
 \left\{\int_{-1}^{1}\frac{\left|t-\frac{n-2}{n}\|x\|_{2}\right|(1-t^{2})^{\frac{n-3}{2}}}{(1-2t\|x\|_{2}+\|x\|_{2}^{2})^{\frac{n-2}{2}}}dt\right\}.
\eeq
For $j\in\{1,\ldots,\nu\}$, let $\psi_{j}(x)=\langle\nabla u_{j}(x),\vartheta\rangle$. Without loss of generality,
we may assume that $\sum_{j=1}^{\nu}|\psi_{j}(x)|^{2}\neq0$. Then  let
\be\label{eq-ch-1-2}\alpha_{j}:=\alpha_{j}(\vartheta)=\frac{|\psi_{j}(x)|^{p-2}\psi_{j}(x)}{\left(\sum_{j=1}^{\nu}|\psi_{j}(x)|^{p}\right)^{\frac{p-1}{p}}}.\ee
It is not difficult to know that $\alpha:=\alpha(\vartheta)\in\partial\mathbf{B}_{\ell_{q}^{\nu}}.$
Hence (\ref{t-0}) follows from (\ref{eq-1.1f}), (\ref{eq-ch-1-1}) and (\ref{eq-ch-1-2}).

Next, we prove the sharpness part of (\ref{t-0}). Let $u=(u_{1},0,\ldots,0)$ be a
harmonic function of $\mathbf{B}_{\ell_{2}^{n}}$ into $\mathbf{B}_{\ell_{p}^{\nu}}$, where
$u_{1}$ is an extremal function  of Theorem \Ref{Liu-2019}.  Then $\|u\|_{p}=|u_{1}|$.
We split the remaining  proof into two cases
\bca\label{cla-2-01}
Let $x\in\mathscr{P}^{\ast}:=\{y\in\mathbf{B}_{\ell_{2}^{n}}:~u_{1}(y)=0\}$. \eca
In this case, by (\ref{eq-1.1f}), we have

\beqq
|\nabla \|u(x)\|_{p}|&=&\limsup_{y\rightarrow x}\frac{|\|u(x)\|_{p}-\|u(y)\|_{p}|}{\|x-y\|_{2}}=
\limsup_{y\rightarrow x}\frac{|u_{1}(x)-u_{1}(y)|}{\|x-y\|_{2}}\\
&=&\|\nabla u_{1}(x)\|_{2}.
\eeqq

\bca
Let $x\in\mathbf{B}_{\ell_{2}^{n}}\setminus\mathscr{P}^{\ast}$. \eca
By (\ref{eq-1.4g}), we see that

\beqq
|\nabla \|u(x)\|_{p}|=\max_{\vartheta\in\partial\mathbf{B}_{\ell_{2}^{n}}}\left|\sum_{j=1}^{n}\frac{\partial |u_{1}(x)|}{\partial x_{j}}\vartheta_{j}\right|
=\max_{\vartheta\in\partial\mathbf{B}_{\ell_{2}^{n}}}\left|
\sum_{j=1}^{n}\frac{\partial u_{1}(x)}{\partial x_{j}}\frac{u_{1}(x)}{|u_{1}(x)|}\vartheta_{j}\right|=\|\nabla u_{1}(x)\|_{2}.
\eeqq
The proof of this theorem is completed.
\qed

\subsection*{ The proof of Theorem \ref{thm-2+}} Let $f=(f_{1},\ldots,f_{\nu})$ be a pluriharmonic function of $\mathbb{B}_{\ell_{2}^{n}}$
into $\mathbb{B}_{\ell_{p}^{\nu}}$.
We divide the proof of (\ref{eq-ch-3.3g}) into two cases.

\bca\label{cla-01}
$n=1$. \eca
We split the proof of this case into two steps.
\bst\label{bst-01} We first estimate $|\nabla \|f(z)\|_{p}|$ for $z=x+iy \in\Omega:=\{\mathcal{Z}\in\mathbb{D}:~\|f(\mathcal{Z})\|_{p}\neq0\}$.  \est

Since
\beqq
\max_{\alpha\in[0,2\pi]}\left|\frac{\partial
\|f(z)\|_{p}}{\partial x}\cos\alpha+\frac{\partial
\|f(z)\|_{p}}{\partial
y}\sin\alpha\right|&=&\frac{1}{2}\bigg(\left|\frac{\partial \|f(z)\|_{p}}{\partial
x}+i\frac{\partial \|f(z)\|_{p}}{\partial
y}\right|\\
&&+\left|\frac{\partial \|f(z)\|_{p}}{\partial
x}-i\frac{\partial \|f(z)\|_{p}}{\partial
y}\right|\bigg),
\eeqq
by Lemma \ref{lem-Schwarz-def}, we see that

\beq\label{en-1} |\nabla
\|f(z)\|_{p}|&=&
\max_{\alpha\in[0,2\pi]}\limsup_{\rho\rightarrow 0^{+}}\frac{|\|f(z+\rho e^{i\alpha})\|_{p}-\|f(z)\|_{p}|}{\rho}\\ \nonumber
&=&\max_{\alpha\in[0,2\pi]}\left|\frac{\partial
\|f(z)\|_{p}}{\partial x}\cos\alpha+\frac{\partial
\|f(z)\|_{p}}{\partial
y}\sin\alpha\right|\\  \nonumber
&=&\left|\frac{\partial \|f(z)\|_{p}}{\partial
z}\right|+\left|\frac{\partial \|f(z)\|_{p}}{\partial
\overline{z}}\right|. \eeq
Next, we estimate $\left|\frac{\partial \|f(z)\|_{p}}{\partial
z}\right|+\left|\frac{\partial \|f(z)\|_{p}}{\partial
\overline{z}}\right|$.
Elementary computations give that

\be\label{c-1.0h} \frac{\partial \|f(z)\|_{p}}{\partial
z}=\frac{1}{2}\|f(z)\|_{p}^{1-p}\sum_{k=1}^{\nu}|f_{k}(z)|^{p-2}\left(\frac{\partial
f_{k}(z)}{\partial
z}\overline{f_{k}(z)}+\overline{\left(\frac{\partial
f_{k}(z)}{\partial \overline{z}}\right)} f_{k}(z)\right)\ee
and \be\label{c-1.1h} \frac{\partial \|f(z)\|_{p}}{\partial
\overline{z}}=\frac{1}{2}\|f(z)\|_{p}^{1-p}\sum_{k=1}^{\nu}|f_{k}(z)|^{p-2}\left(\frac{\partial
f_{k}(z)}{\partial
\overline{z}}\overline{f_{k}(z)}+\overline{\left(\frac{\partial
f_{k}(z)}{\partial z}\right)} f_{k}(z)\right).\ee
Let  \be\label{c-1.2h}\xi(z)=\left(\frac{|f_{1}(z)|^{p-2}f_{1}(z)}{\|f(z)\|_{p}^{p-1}},\cdots,\frac{|f_{\nu}(z)|^{p-2}f_{\nu}(z)}{\|f(z)\|_{p}^{p-1}}\right).\ee
Then $\xi(z)\in\partial\mathbb{B}_{\ell_{q}^{\nu}}$, where $\frac{1}{q}+\frac{1}{p}=1$.
It follows from (\ref{c-1.0h}), (\ref{c-1.1h}) and (\ref{c-1.2h}) that
$$\frac{\partial \|f(z)\|_{p}}{\partial
z}=\frac{1}{2}\left(\langle\partial f(z), \xi(z)\rangle+\overline{\langle\overline{\partial} f(z), \xi(z)\rangle}\right)$$
and
$$\frac{\partial \|f(z)\|_{p}}{\partial
\overline{z}}=\frac{1}{2}\left(\langle\overline{\partial} f(z), \xi(z)\rangle+\overline{\langle\partial f(z), \xi(z)\rangle}\right),$$
where $\partial f(z)=\left(\frac{\partial f_{1}(z)}{\partial z},\cdots,\frac{\partial f_{\nu}(z)}{\partial z}\right)$,
$\overline{\partial} f(z)=\left(\frac{\partial f_{1}(z)}{\partial \overline{z}},\cdots,\frac{\partial f_{\nu}(z)}{\partial \overline{z}}\right)$
and $\langle\cdot,\cdot\rangle$ is the Euclidean inner product on $\mathbb{C}^{n}$.
Hence we obtain the following inequality
\be\label{c-1.3h}\left|\frac{\partial \|f(z)\|_{p}}{\partial
z}\right|+\left|\frac{\partial \|f(z)\|_{p}}{\partial
\overline{z}}\right|\leq|\langle\partial f(z), \xi(z)\rangle|+|\langle\overline{\partial} f(z), \xi(z)\rangle|.\ee
Now we estimate $|\langle\partial f(z), \xi(z)\rangle|+|\langle\overline{\partial} f(z), \xi(z)\rangle|.$
For any fixed $\theta=(\theta_{1},\ldots,\theta_{\nu})\in\partial\mathbb{B}_{\ell_{q}^{\nu}}$, let $F_{\theta}(w)=\langle f(w), \theta\rangle,~w\in \mathbb{D}$.
Then by H\"older's inequality, we have

$$|F_{\theta}(w)|\leq|\langle f(w), \theta\rangle|\leq\left(\sum_{k=1}^{\nu}|f_{k}(w)|^{p}\right)^{\frac{1}{p}}
\left(\sum_{k=1}^{\nu}|\theta_{k}|^{q}\right)^{\frac{1}{q}}<1.$$
This implies that $F_{\theta}$ is a harmonic function of $\mathbb{D}$ into itself.
For any fixed $\theta\in\partial\mathbb{B}_{\ell_{q}^{\nu}}$,  an application of (\ref{Col-89}) to $F_{\theta}$ gives

\be\label{c-1.4h} \left|\frac{\partial F_{\theta}(z)}{\partial z}\right|+
\left|\frac{\partial F_{\theta}(z)}{\partial \overline{z}}\right|=
|\langle\partial f(z),\theta\rangle|+|\langle\overline{\partial} f(z), \theta\rangle|\leq\frac{4}{\pi}\frac{1}{1-|z|^{2}}.\ee

Combining (\ref{en-1}), (\ref{c-1.3h}) and (\ref{c-1.4h}), we conclude that

$$|\nabla\|f(z)\|_{p}|=\left|\frac{\partial \|f(z)\|_{p}}{\partial
z}\right|+\left|\frac{\partial \|f(z)\|_{p}}{\partial
\overline{z}}\right|\leq\frac{4}{\pi}\frac{1}{1-|z|^{2}}.$$

\bst\label{bst-01b}
Next, we estimate $|\nabla \|f(z)\|_{p}|$ for $z=x+iy \in\mathbb{D}\backslash\Omega$.
\est
%
By Lemma \ref{lem-Schwarz-def}, we see that
%
%
%
\beq\label{eq-gh-01}
|\nabla\|f(z)\|_{p}|
&=&\max_{\alpha\in[0,2\pi]}
\left\{\lim_{\rho\rightarrow0^{+}}\frac{\big|\|f(z+\rho e^{i\alpha})\|_{p}-\|f(z)\|_{p}\big|}{\rho}\right\}\\ \nonumber
&=&\max_{\alpha\in[0,2\pi]}\left\{\lim_{\rho\rightarrow0^{+}}\frac{\|f(z+\rho e^{i\alpha})-f(z)\|_{p}}{\rho}\right\}\\ \nonumber
&=&\max_{\alpha\in[0,2\pi]}\left(\sum_{k=1}^{\nu}\left|\frac{\partial f_{k}(z)}{\partial x}\cos\alpha+
\frac{\partial f_{k}(z)}{\partial y}\sin\alpha\right|^{p}\right)^{\frac{1}{p}}.
\eeq
For any fixed $\theta\in\partial\mathbb{B}_{\ell_{q}^{\nu}}$, let $F_{\theta}(w)=\langle f(w), \theta\rangle,~w\in\mathbb{D}$.
Then $F_{\theta}$ is a harmonic function of $\mathbb{D}$ into itself.
For all $\alpha\in [0,2 \pi]$, we have

\beqq
\lim_{\rho\rightarrow0^{+}}\frac{\left|F_{\theta}(z+\rho e^{i\alpha})-F_{\theta}(z)\right|}{\rho}
&=&
\left|\frac{\partial F_{\theta}(z)}{\partial x}\cos\alpha+\frac{\partial F_{\theta}(z)}{\partial y}\sin\alpha\right|\\
&=&
\left|e^{i\alpha}\frac{\partial F_{\theta}(z)}{\partial z}+
e^{-i\alpha}\frac{\partial F_{\theta}(z)}{\partial \overline{z}}\right|\\
&\leq&
\left|\frac{\partial F_{\theta}(z)}{\partial z}\right|+
\left|\frac{\partial F_{\theta}(z)}{\partial
\overline{z}}\right|
\eeqq
and
\beqq
\left|\sum_{k=1}^{\nu}\left(\frac{\partial f_{k}(z)}{\partial
x}\cos\alpha+ \frac{\partial f_{k}(z)}{\partial
y}\sin\alpha\right)\overline{\theta}_{k}\right|&=&
\lim_{\rho\rightarrow0^{+}}\left|\sum_{k=1}^{\nu}\frac{(f_{k}(z+\rho e^{i\alpha})-f_{k}(z))}{\rho}\overline{\theta}_{k}\right|\\
&=&\lim_{\rho\rightarrow0^{+}}\frac{\left|F_{\theta}(z+\rho e^{i\alpha})-F_{\theta}(z)\right|}{\rho},
\eeqq
which, together with (\ref{Col-89}), imply that
\be\label{eq-gh-02}\left|\sum_{k=1}^{\nu}\left(\frac{\partial
f_{k}(z)}{\partial x}\cos\alpha+ \frac{\partial f_{k}(z)}{\partial
y}\sin\alpha\right)\overline{\theta}_{k}\right|\leq\left|\frac{\partial F_{\theta}(z)}{\partial z}\right|+
\left|\frac{\partial F_{\theta}(z)}{\partial
\overline{z}}\right|\leq\frac{4}{\pi}\frac{1}{1-|z|^{2}}.
\ee
For $k\in\{1,2,\cdots,\nu\}$, let $\mu_{k}(z)=\frac{\partial
f_{k}(z)}{\partial x}\cos\alpha+\frac{\partial f_{k}(z)}{\partial
y}\sin\alpha$. Without loss of generality, we assume $\sum_{k=1}^{\nu}|\mu_{k}(z)|^{p}\neq0$. Then we let $$\theta_{k}:=\theta_{k}(\alpha)=
\frac{|\mu_{k}(z)|^{p-2}\mu_{k}(z)}
{\left(\sum_{k=1}^{\nu}|\mu_{k}(z)|^{p}\right)^{\frac{p-1}{p}}}.$$
It is not difficult  to know that $\theta(\alpha)=(\theta_1(\alpha),\dots, \theta_{\nu}(\alpha))\in \partial\mathbb{B}_{\ell_{q}^{\nu}}$.
From (\ref{eq-gh-01}) and (\ref{eq-gh-02}), we conclude that for
$z\in\mathbb{D}\backslash\Omega$,
\beqq|\nabla\|f(z)\|_{p}|&=&\max_{\alpha\in[0,2\pi]}\left(\sum_{k=1}^{\nu}\left|\frac{\partial
f_{k}(z)}{\partial x}\cos\alpha+ \frac{\partial f_{k}(z)}{\partial
y}\sin\alpha\right|^{p}\right)^{\frac{1}{p}}\\
&=&\max_{\alpha\in[0,2\pi]}\left|\sum_{k=1}^{\nu}\left(\frac{\partial
f_{k}(z)}{\partial x}\cos\alpha+ \frac{\partial f_{k}(z)}{\partial
y}\sin\alpha\right)\overline{\theta}_{k}(\alpha)\right|\\
&\leq&\frac{4}{\pi}\frac{1}{1-|z|^{2}}. \eeqq

\bca\label{cla-02}
$n\geq2$. \eca

We also split the proof of this case into two steps.
\bst\label{bst-03} We first estimate $|\nabla \|f(0)\|_{p}|$. \est

For any fixed $\theta\in \partial \mathbb{B}_{\ell_{2}^{n}}$,
let
\[
\Phi_{\theta}(\zeta)=f(\zeta \theta),
\quad
\zeta \in \mathbb{D}.
\]
Then $\Phi_{\theta}$ is a harmonic function of  $\mathbb{D}$ into $\mathbb{B}_{\ell_{p}^{\nu}}$.
By the case \ref{cla-01}, we have
\[
|\nabla \|\Phi_{\theta}(0)\|_{p}|\leq\frac{4}{\pi}.
\]
Since $\theta\in \partial \mathbb{B}_{\ell_{2}^{n}}$ is arbitrary,
in view of Lemma \ref{lem-Schwarz-def}, we have
\begin{equation}
\label{eq-ch-3.3g_0}
|\nabla \|f(0)\|_{p}|\leq\frac{4}{\pi}.
\end{equation}

\bst\label{bst-04} Next, we estimate $|\nabla \|f(z)\|_{p}|$ for $z\in \mathbb{B}_{\ell_{2}^{n}}\setminus \{ 0\}$.\est

In this case, by \cite[Theorem 2.2.2]{R2},
there exists an automorphism $\varphi_z$ of $\mathbb{B}_{\ell_{2}^{n}}$ such that
$\varphi_z(0)=z$ and
\[
D(\varphi_z^{-1})(z)(w)=
-\frac{\langle w,z\rangle z}{\| z\|_2^2(1-\| z\|_2^2)}-\frac{\| z\|_2^2w-\langle w,z\rangle z}{\| z\|_2^2\sqrt{1-\| z\|_2^2}},
\]
where $D(\varphi_z^{-1})(z)$ is the Fr\'{e}chet derivative of $\varphi_z^{-1}$ at $z$.
Since $f\circ \varphi_z$ is a pluriharmonic function of $\mathbb{B}_{\ell_{2}^{n}}$
into $\mathbb{B}_{\ell_{p}^{\nu}}$ and
\[
\| D(\varphi_z^{-1})(z)\|=\sup_{\| w\|_2=1}\| D(\varphi_z^{-1})(z)(w)\|_2=\frac{1}{1-\| z\|_2^2},
\]
by using (\ref{eq-ch-3.3g_0}),
we have
\beqq
|\nabla \|f(z)\|_{p}|&=&\limsup_{w\rightarrow z}\frac{|\|f(z)\|_{p}-\|f(w)\|_{p}|}{\|z-w\|_{2}}
\\
&=&
\limsup_{w\rightarrow z}
\frac{|\|f\circ \varphi_z(0)\|_{p}-\|f\circ \varphi_z(\varphi_z^{-1}(w))\|_{p}|}{\|\varphi_z^{-1}(w)\|_2}
\cdot
\frac{\|\varphi_z^{-1}(w)-\varphi_z^{-1}(z)\|_2}{\|z-w\|_{2}}
\\
&\leq &
|\nabla \|f\circ \varphi_z(0)\|_{p}|\cdot \| D(\varphi_z^{-1})(z)\|
\\
&\leq &
\frac{4}{\pi}\frac{1}{1-\|z\|_{2}^{2}}.
\eeqq

Therefore, combining the cases \ref{cla-01} and \ref{cla-02} yields the final estimate (\ref{eq-ch-3.3g}).

Next, we prove the sharpness part.
Let $a\in \mathbb{B}_{\ell_{2}^{n}}$ be arbitrarily fixed.
There exists a unitary transformation $U$ such that $Ua=(b_1, 0,\dots, 0)$ for some $b_1\in \mathbb{R}$ with $b_1 \in [0,1)$.
Obviously,
 \be\label{Sharp-1} \frac{1}{1-|b_1|^2}
=\frac{1}{1-\| a\|_2^2}.\ee
Let $f: \mathbb{B}_{\ell_{2}^{n}}\to \mathbb{B}_{\ell_{p}^{\nu}}$
be such that
$$\tilde{f}(z)=(f\circ U^{-1})(z)=\left((g\circ \phi_{-b_1})(z_1),0,\ldots,0\right)=(\tilde{f}_{1}(z_{1}),0,\ldots,0),$$
where $g(\zeta)=\frac{2}{\pi}\arctan\frac{i(\overline{\zeta}-\zeta)}{1-|\zeta|^{2}}$
and
$\phi_{-b_1}(\zeta)=\frac{-b_1+\zeta}{1-{b_1}\zeta}$
is a conformal automorphism of $\mathbb{D}$.
Since $f(a)=0$
and
\beqq
\limsup_{w\rightarrow a}\frac{|\|\tilde{f}(Uw)\|_{p}-\|\tilde{f}(Ua)\|_{p}|}{\|Uw-Ua\|_{2}}
&=&
\sup_{|\theta_1|=1}\limsup_{\rho\rightarrow 0^+}\frac{|\tilde{f}_1(b_1+\rho\theta_1)-\tilde{f}_1(b_1)|}{\rho}
\\
&=&\left|\frac{\partial \tilde{f}_{1}(b_1)}{\partial z_{1}}\right|+\left|\frac{\partial \tilde{f}_{1}(b_1)}{\partial \overline{z}_{1}}\right|
\\
&=&
\frac{4}{\pi}\frac{1}{1-|b_1|^2},
\eeqq
by (\ref{Sharp-1}),
we conclude that
\beqq
|\nabla \|f(a)\|_{p}|&=&\limsup_{w\rightarrow a}\frac{|\|f(w)\|_{p}-\|f(a)\|_{p}|}{\|w-a\|_{2}}
=
\limsup_{w\rightarrow a}\frac{|\|\tilde{f}(Uw)\|_{p}-\|\tilde{f}(Ua)\|_{p}|}{\|Uw-Ua\|_{2}}
\\
&=&\frac{4}{\pi}\frac{1}{1-\| a\|_2^2}.
\eeqq
The proof of this theorem is completed.
\qed

\subsection*{ The proof of Theorem \ref{thm-1}} Let $u=(u_{1},\ldots,u_{\nu})$ be a pluriharmonic function of $\mathbb{B}_{\ell_{2}^{n}}$
 into $\mathbf{B}_{\ell_{p}^{\nu}}$. We also divide the proof of (\ref{eq-th-1}) into two cases.
 \bca\label{cla-1}
$n=1$. \eca
We split the proof of this case into two steps.

 \bst\label{bst-1} We first estimate $\big|\nabla \|u(z)\|_{p}\big|$ for $z=x+iy\in\Omega:=\{\zeta\in\mathbb{D}:~u(\zeta)\neq0\}$.
 \est
As in the proof of Theorem \ref{thm-2+}, we have
\be\label{eq-1.1}
|\nabla \|u(z)\|_{p}|=
\left|\frac{\partial\|u(z)\|_{p}}{\partial z}\right|+\left|\frac{\partial\|u(z)\|_{p}}{\partial \overline{z}}\right|.
\ee

By calculations, we obtain
\beqq
\frac{\partial\|u(z)\|_{p}}{\partial z}&=&\sum_{j=1}^{\nu}\frac{|u_j(z)|^{p-2}u_{j}(z)}{\|u(z)\|_{p}^{p-1}}
\frac{\partial u_{j}(z)}{\partial z}
\\
&=&
\frac{1}{2}\left\langle\tau(z),\frac{\partial u(z)}{\partial x}\right\rangle-
\frac{i}{2}\left\langle\tau(z),\frac{\partial u(z)}{\partial y}\right\rangle
\eeqq
and
$$
\frac{\partial\|u(z)\|_{p}}{\partial \overline{z}}
=\frac{1}{2}\left\langle\tau(z),\frac{\partial u(z)}{\partial x}\right\rangle+
\frac{i}{2}\left\langle\tau(z),\frac{\partial u(z)}{\partial y}\right\rangle,
$$
which, together with (\ref{eq-1.1}), implies that
\be\label{eq-1.2}
|\nabla \|u(z)\|_{p}|=\sqrt{\left|\left\langle\frac{\partial u(z)}{\partial x},\tau(z)\right\rangle\right|^{2}+
\left|\left\langle\frac{\partial u(z)}{\partial y},\tau(z)\right\rangle\right|^{2}},
\ee
where
\[
\tau(z)=\left(\frac{|u_1(z)|^{p-2}u_{1}(z)}{\|u(z)\|_{p}^{p-1}},\dots,
\frac{|u_{\nu}(z)|^{p-2}u_{\nu}(z)}{\|u(z)\|_{p}^{p-1}}\right),
\]
$\frac{\partial u(z)}{\partial x}=\left(\frac{\partial u_{1}(z)}{\partial x},\ldots,\frac{\partial u_{\nu}(z)}{\partial x}\right)$ and
 $\frac{\partial u(z)}{\partial y}=\left(\frac{\partial u_{1}(z)}{\partial y},\ldots,\frac{\partial u_{\nu}(z)}{\partial y}\right).$
Note that $\tau(z)\in \partial \mathbf{B}_{\ell_{q}^{\nu}}$,
where $\frac{1}{q}+\frac{1}{p}=1$.
For any fixed $\theta\in\mathbf{B}_{\ell_{q}^{\nu}}$, let $U_{\theta}(\varsigma)=\langle u(\varsigma),\theta\rangle$, $\varsigma\in\mathbb{D}$.
Then we infer from  H\"older's inequality that $U_{\theta}$ is a real harmonic function of $\mathbb{D}$ into $(-1,1)$.
Hence, by (\ref{KV-2011}), we have
\beqq
\|\nabla U_{\theta}(\varsigma)\|_2=\sqrt{\left|\left\langle\frac{\partial u(\varsigma)}{\partial x},\theta\right\rangle\right|^{2}+
\left|\left\langle\frac{\partial u(\varsigma)}{\partial y},\theta\right\rangle\right|^{2}}
\leq\frac{4}{\pi}\frac{1-|U_{\theta}(\varsigma)|^{2}}{1-|\varsigma|^{2}},
\eeqq
which, together with (\ref{eq-1.2}), gives that
$$|\nabla \|u(z)\|_{p}|\leq\frac{4}{\pi}\frac{1-|U_{\tau(z)}(z)|^{2}}{1-|z|^{2}}
=\frac{4}{\pi}\frac{1-\|u(z)\|_{p}^{2}}{1-|z|^{2}}.$$

 \bst\label{bst-2} Next, we  estimate $\big|\nabla \|u(z)\|_{p}\big|$ for $z\in\mathbb{D}\setminus\Omega$.
 \est
%
This step follows from Theorem \ref{thm-2+}.

\bca\label{cla-2}
$n\geq2$. \eca
We also divide the proof of this case into two steps.
 \bst\label{bst-3} We first estimate $|\nabla \|u(0)\|_{p}|$.
 \est
For any fixed $\theta\in\partial\mathbb{B}_{\ell_{2}^{n}}$, let $\psi_{\theta}(\xi)=u(\xi\theta)$, $\xi\in\mathbb{D}$.
Obviously, $\psi_{\theta}$ is a harmonic function of $\mathbb{D}$ into $\mathbf{B}_{\ell_{p}^{\nu}}$. By Case \ref{cla-1},
we have
$$|\nabla \|\psi_{\theta}(0)\|_{p}|\leq\frac{4}{\pi}\left(1-\|\psi_{\theta}(0)\|_{p}^{2}\right)=
\frac{4}{\pi}\left(1-\|u(0)\|_{p}^{2}\right).$$ It follows from Lemma \ref{lem-Schwarz-def}
and the arbitrariness of $\theta$ that
\be\label{eq-1.3f}|\nabla \|u(0)\|_{p}|\leq
\frac{4}{\pi}\left(1-\|u(0)\|_{p}^{2}\right).\ee

 \bst\label{bst-4} Next, we estimate $|\nabla \|u(z)\|_{p}|$ for $z\in\mathbb{B}_{\ell_{2}^{n}}\setminus\{0\}$.
 \est

By using (\ref{eq-1.3f}) and  the similar reasoning as in the proof of Step \ref{bst-04} of Theorem \ref{thm-2+},
we obtain that

\beqq
|\nabla \|u(z)\|_{p}|\leq\frac{4}{\pi}\frac{1-\|u(z)\|_{p}^{2}}{1-\|z\|_{2}^{2}}.
\eeqq

Combining  Cases \ref{cla-1} and \ref{cla-2} yields the final estimate (\ref{eq-th-1}).

To show that the inequality of {\rm(\ref{eq-th-1})} is sharp,
let $a\in \mathbb{B}_{\ell_{2}^{n}}$ be arbitrarily fixed.
There exists a unitary transformation $U$ such that $Ua=(a_1^{\ast}, 0,\dots, 0)$ for some $a_1^{\ast}\in \mathbb{R}$ with $a_1^{\ast} \in [0,1)$.
Let $u: \mathbb{B}_{\ell_{2}^{n}}\to \mathbf{B}_{\ell_{p}^{\nu}}$
be a harmonic function such that
$$\tilde{u}(z):=(u\circ U^{-1})(z)=\left((g\circ \phi_{-a_1^{\ast}})(z_1),0,\ldots,0\right)=(\tilde{u}_{1}(z_{1}),0,\ldots,0),$$
where $g(\zeta)=\frac{2}{\pi}\arctan\frac{i(\overline{\zeta}-\zeta)}{1-|\zeta|^{2}}$
and
$\phi_{-a_1^{\ast}}(\zeta)=\frac{-a_1^{\ast}+\zeta}{1-{a_1^{\ast}}\zeta}$
is a conformal automorphism of $\mathbb{D}$.
Since $u(a)=0$, as in the proof of Theorem \ref{thm-2+},
we have
\[
|\nabla \|u(a)\|_{p}|=\frac{4}{\pi}\frac{1-\| u(a)\|_p^2}{1-\| a\|_2^2}.
\]
The proof of this theorem is complete.
\qed

\subsection*{ The proof of Theorem \ref{thm-3+}} Since (\ref{qq-ch1.01g}) easily follows from (\ref{qq-ch1.1g}), we
 only need to prove (\ref{qq-ch1.1g}).
Let $\beta=(\beta_{1},\ldots,\beta_{n})$ be a multi-index consisting of $n$ nonnegative integers $\beta_{k}$, where $k\in\{1,\ldots,n\}$.
For $z=(z_{1},\ldots,z_{n})\in\mathbb{B}_{\ell_{\infty}^{n}}$, we write $f$ in the following form
$$f(z)=\sum_{\beta}a_{\beta}z^{\beta}+\sum_{\beta}\overline{b}_{\beta}\overline{z}^{\beta},$$
where $a_{\beta}=(a_{1,\beta},\ldots,a_{\nu,\beta})$,  $b_{\beta}=(b_{1,\beta},\ldots,b_{\nu,\beta})$
and $f_{j}(z)=\sum_{\beta}a_{j,\beta}z^{\beta}+\sum_{\beta}\overline{b}_{j,\beta}\overline{z}^{\beta}$ for $j\in\{1,\ldots,\nu\}$.
For any
$\zeta=(\zeta_{1},\ldots,\zeta_{n})\in\mathbb{B}_{\ell_{\infty}^{n}},$
let
$\zeta\otimes\theta:=(\zeta_{1}e^{i\theta_{1}},\ldots,\zeta_{n}e^{i\theta_{n}})$,
where  $\theta_{k}\in[0,2\pi]$ for $k\in\{1,\ldots,n\}$.
Then

$$
\frac{1}{(2\pi)^{n}}\int_{0}^{2\pi}\cdots\int_{0}^{2\pi}\left\|f(\zeta\otimes\theta)\right\|^{2}_{2}d\theta_{1}\cdots
d\theta_{n}
 =\|f(0)\|^{2}_{2} + \sum_{|\beta|=1}^{\infty}\left(\|a_{\beta}\|^{2}_{2}+\|b_{\beta}\|^{2}_{2}\right)\big|\zeta^{\beta}\big|^{2}
 \leq1.
$$
By letting $\zeta\rightarrow\partial\mathbb{B}_{\ell_{\infty}^{n}}$, we get
\beqq
\sum_{|\beta|=1}\left(\|a_{\beta}\|^{2}_{2}+\|b_{\beta}\|^{2}_{2}\right)\leq
\sum_{|\beta|=1}^{\infty}\left(\|a_{\beta}\|^{2}_{2}+\|b_{\beta}\|^{2}_{2}\right)\leq1-\|f(0)\|^{2}_{2},
\eeqq
which implies that

\be\label{eq-h-gg}
\sum_{j=1}^{\nu}\sum_{k=1}^{n}\left(\left|\frac{\partial f_{j}(0)}{\partial z_{k}}\right|^{2}+
\left|\frac{\partial f_{j}(0)}{\partial \overline{z}_{k}}\right|^{2}\right)\leq 1-\|f(0)\|_{2}^{2}.
\ee
For any fixed $a=(a_{1},\ldots,a_{n})\in\mathbb{B}_{\ell_{\infty}^{n}}$,
let $$F(\zeta)=(F_{1}(\zeta),\ldots,F_{\nu}(\zeta))=f(\phi_{a}(\zeta)),$$ where
$\phi_{a}(\zeta)=(\phi_{a_{1}}(\zeta_{1}),\ldots,\phi_{a_{n}}(\zeta_{n}))$
is a holomorphic automorphism of $\mathbb{B}_{\ell_{\infty}^{n}}$ with  $\phi_{a_{k}}(\zeta_{k})=\frac{a_{k}+\zeta_{k}}{1+\overline{a}_{k}\zeta_{k}}$
 for
$k\in\{1,\ldots,n\}$.
Elementary calculations show that
$$\frac{\partial F_{j}(\zeta)}{\partial \zeta_{k}}=\frac{\partial f_{j}(\phi_{a}(\zeta))}{\partial w_{k}}\phi_{a_{k}}'(\zeta_{k})
=\frac{\partial f_{j}(\phi_{a}(\zeta))}{\partial w_{k}}\frac{1-|a_{k}|^{2}}{(1+\overline{a}_{k}\zeta_{k})^{2}}$$
and
$$\frac{\partial F_{j}(\zeta)}{\partial \overline{\zeta}_{k}}=\frac{\partial f_{j}(\phi_{a}(\zeta))}{\partial \overline{w}_{k}}\overline{\phi_{a_{k}}'(\zeta_{k})}
=\frac{\partial f_{j}(\phi_{a}(\zeta))}{\partial \overline{w}_{k}}\frac{1-|a_{k}|^{2}}{(1+a_{k}\overline{\zeta}_{k})^{2}},$$
which, together with (\ref{eq-h-gg}), imply
\beqq\label{eq-1.6g}
\sum_{j=1}^{\nu}\sum_{k=1}^{n}\left(\left|\frac{\partial f_{j}(a)}{\partial w_{k}}\right|^{2}+
\left|\frac{\partial f_{j}(a)}{\partial\overline{w}_{k}}\right|^{2}\right)(1-|a_{k}|^{2})^{2}
&=&\sum_{j=1}^{\nu}\sum_{k=1}^{n}\left(\left|\frac{\partial F_{j}(0)}{\partial \zeta_{k}}\right|^{2}+
\left|\frac{\partial F_{j}(0)}{\partial \overline{\zeta}_{k}}\right|^{2}\right)\\
&\leq&1-\|F(0)\|_{2}^{2}\\
&=&1-\|f(a)\|_{2}^{2},\eeqq
where $w_{k}=\phi_{a_{k}}(\zeta_{k}).$

Next, we prove the sharpness part of (\ref{qq-ch1.1g}).
Let $a\in \mathbb{B}_{\ell_{\infty}^{n}}$ be arbitrarily fixed.
For $z\in\mathbb{B}_{\ell_{\infty}^{n}}$, let $f(z)=(f_{1}(z),f_{2}(z),\ldots,f_{\nu}(z))\in\mathbb{B}_{\ell_{2}^{\nu}}$,
where $f_{1}(z)=\frac{-a_{1}+z_{1}}{1-\overline{a}_{1}z_{1}}$ and $f_{j}(z)\equiv0$ for $j\in\{2,\ldots,\nu\}$.
Then
\beqq
\sum_{j=1}^{\nu}\sum_{k=1}^{n}\left(\left|\frac{\partial f_{j}(a)}{\partial z_{k}}\right|^{2}+
\left|\frac{\partial f_{j}(a)}{\partial\overline{z}_{k}}\right|^{2}\right)(1-| a_{k}|^2)^2
=1-\| f(a)\|_2^2.\eeqq
The proof of this
theorem is completed.
\qed

\subsection*{ The proof of Proposition \ref{prop-1+}} For any fixed point $z\in\Omega$, we consider the function
$$F(\epsilon)=f(z+d(z)\epsilon)/M_{z},~\epsilon\in\mathbb{D},$$ where $d(z):=d_{\Omega}(z)$ and $M_{z}:=\sup\{\|f(\iota)\|_{2}:~|\iota-z|<d(z)\}$.
Then $F$ is a harmonic mapping of $\mathbb{D}$ into $\mathbb{B}_{\ell_{2}^{n}}$,
$$\partial F(\epsilon)=d(z)\partial f(z+d(z)\epsilon)/M_{z}$$ and $$\overline{\partial} F(\epsilon)=d(z)\overline{\partial} f(z+d(z)\epsilon)/M_{z}.$$
It follows from Theorem \ref{thm-3+} that
$$\left(\|\partial F(0)\|_{2}^{2}+\|\overline{\partial} F(0)\|_{2}^{2}\right)^{\frac{1}{2}}\leq\left(1-\| F(0)\|_{2}^{2}\right)^{\frac{1}{2}},$$
which implies
\be\label{eq-1.5h}d(z)\left(\|\partial f(z)\|_{2}^{2}+\|\overline{\partial}
f(z)\|_{2}^{2}\right)^{\frac{1}{2}}\leq \left(M_{z}^{2}-\|f(z)\|_{2}^{2}\right)^{\frac{1}{2}}.\ee
Let $\eta\in\mathbb{D}_{z}(d(z)):=\{\lambda\in\mathbb{C}:~|\lambda-z|<d(z)\}$.
It follows from the assumption $\|f\|_{2}^2\in\Lambda_{2}(\Omega)$ that there is a positive constant $C$ such that
\beqq
\|f(\eta)\|_{2}^{2}-\|f(z)\|_{2}^{2}
\leq C|\eta-z|^{2}
\leq C(d(z))^{2},
\eeqq
which yields
\be\label{eq-1.6h}
M_{z}^{2}-\|f(z)\|_{2}^{2}=\sup_{\eta\in\mathbb{D}(z,d(z))}\left(\|f(\eta)\|_{2}^{2}-\|f(z)\|_{2}^{2}\right)\leq
C(d(z))^{2}.\ee

Combining (\ref{eq-1.5h}) and (\ref{eq-1.6h}) gives the following estimate
\be\label{eq-1.7h}\left(\|\partial f(z)\|_{2}^{2}+\|\overline{\partial}
f(z)\|_{2}^{2}\right)^{\frac{1}{2}}\leq C^{\frac{1}{2}}.\ee
Finally, since $\Omega$ is linearly connected,
there exists a constant $L>0$ such that given any two points $z_{1},z_{2}\in\Omega$,
there exists a smooth curve
$\gamma\subset\Omega$ with the endpoints $z_{1}$ and $z_{2}$ such that $\ell(\gamma)\leq
L| z_1-z_2|$.
Then, by (\ref{eq-1.7h}) and the Cauchy-Schwarz inequality, we have

$$
\|f(z_{1})-f(z_{2})\|_{2}\leq\int_{\gamma}
\left(\|\partial f(z)\|_{2}+\|\overline{\partial} f(z)\|_{2}\right)|dz|
\leq(2C)^{\frac{1}{2}}L|z_{1}-z_{2}|.
$$ The proof of this proposition is completed.
\qed

\section{The coefficient type Schwarz-Pick  lemmas and their applications}\label{csw-sec3}

We begin this part with the following useful lemma.

\begin{lem}{\rm (\cite[Lemma 1]{CR})}\label{lem-CR}
Let $m$ be a positive integer and $\gamma$ be a real constant. Then
\[
\int_{0}^{2\pi}|\cos (m\theta+\gamma)|d\theta=4.
\]
\end{lem}

\subsection*{ The proof of Theorem \ref{lem-3.1}}
We first prove (\ref{c-1.7s}). Let
$f=h+\overline{g}\in\mathscr{PH}(\mathbb{B}_{\ell_{p}^{n}})$ with
$\sup_{z\in\mathbb{B}_{\ell_{p}^{n}}}|f(z)|\leq 1,$  where
$h(z)=\sum_{\alpha}a_{\alpha}z^{\alpha}$ and
$g(z)=\sum_{\alpha}b_{\alpha}z^{\alpha}.$ For any
$$\xi=(\xi_{1},\ldots,\xi_{n})=(\rho_{1}e^{i\mu_{1}},\ldots,\rho_{n}e^{i\mu_{n}})\in\mathbb{B}_{\ell_{p}^{n}},$$
let
$\xi\otimes\theta:=(\xi_{1}e^{i\theta_{1}},\ldots,\xi_{n}e^{i\theta_{n}})$,
where $\mu_{k},\theta_{k}\in\mathbb{R}$ and $\rho_{k}=|\xi_{k}|$ for
all $k\in\{1,\ldots,n\}$. Then
$\xi\otimes\theta\in\mathbb{B}_{\ell_{p}^{n}}.$ Let $\alpha$ with
$|\alpha|\geq 1$ be fixed. Without loss of generality, we may assume
that $|a_{\alpha}||b_{\alpha}|\neq0$. By the orthogonality, we have

$$|a_{\alpha}|\left|\xi^{\alpha}\right|=\frac{1}{(2\pi)^{n}}\int_{0}^{2\pi}\cdots\int_{0}^{2\pi}e^{-i\left(\arg a_{\alpha} +\sum_{k=1}^{n}\alpha_{k}
(\mu_{k}+\theta_{k})\right)}f(\xi\otimes\theta)d\theta_{1}\cdots\,d\theta_{n}$$
and
$$|b_{\alpha}|\left|\xi^{\alpha}\right|=\frac{1}{(2\pi)^{n}}\int_{0}^{2\pi}\cdots\int_{0}^{2\pi}e^{i\left(\arg b_{\alpha} +\sum_{k=1}^{n}\alpha_{k}
(\mu_{k}+\theta_{k})\right)}f(\xi\otimes\theta)d\theta_{1}\cdots\,d\theta_{n},$$
which give that

\begin{equation}\label{eq-c-1.26}
\begin{split}
(|a_{\alpha}|+|b_{\alpha}|)\left|\xi^{\alpha}\right|&=\bigg|\frac{1}{(2\pi)^{n}}\int_{0}^{2\pi}\cdots\int_{0}^{2\pi}\Big(e^{-i\left(\arg
a_{\alpha} +\sum_{k=1}^{n}\alpha_{k} (\mu_{k}+\theta_{k})\right)}\\
&+e^{i\left(\arg b_{\alpha} +\sum_{k=1}^{n}\alpha_{k}
(\mu_{k}+\theta_{k})\right)}\Big)
f(\xi\otimes\theta)d\theta_{1}\cdots\,d\theta_{n}\bigg|\\
&\leq\frac{1}{(2\pi)^{n}}\int_{0}^{2\pi}\cdots\int_{0}^{2\pi}\Big|1+e^{i\left(\arg
a_{\alpha}+\arg b_{\alpha} +2\sum_{k=1}^{n}\alpha_{k}
(\mu_{k}+\theta_{k})\right)}\Big|\\
&\times|f(\xi\otimes\theta)|d\theta_{1}\cdots\,d\theta_{n}\\
&\leq\frac{2}{(2\pi)^{n}}\int_{0}^{2\pi}\cdots\int_{0}^{2\pi}\bigg|\cos\bigg(\sum_{k=1}^{n}\alpha_{k}
(\mu_{k}+\theta_{k})\\
&+\frac{\arg a_{\alpha}+\arg
b_{\alpha}}{2}\bigg)\bigg|d\theta_{1}\cdots\,d\theta_{n}.
\end{split}
\end{equation}
Since $|\alpha|\geq1$, without loss of generality, we may assume
that $\alpha_{1}\neq0$. It follows from Lemma \ref{lem-CR} that

\[
\int_{0}^{2\pi}\bigg|\cos\bigg(\sum_{k=1}^{n}\alpha_{k}
(\mu_{k}+\theta_{k})+\frac{\arg a_{\alpha}+\arg
b_{\alpha}}{2}\bigg)\bigg|d\theta_{1}=4,
\]
which, together
  (\ref{eq-c-1.26}), implies that


\be\label{chen-1-1} |a_{\alpha}|+|b_{\alpha}|\leq
\frac{4}{\pi}\inf_{\xi\in\mathbb{B}_{\ell_{p}^{n}}}\frac{1}{\left|\xi^{\alpha}\right|}.
\ee

From \cite[p.43]{Di}, we see that

\begin{equation}\label{eq-c-1.27t}
\sup_{\xi\in\mathbb{B}_{\ell_{p}^{n}}}|\xi^{\alpha}|=\left(\frac{\alpha^{\alpha}}{|\alpha|^{|\alpha|}}\right)^{1/p}.
\end{equation}

Therefore, combining (\ref{chen-1-1}) and (\ref{eq-c-1.27t}) yields the final estimates

$$|a_{\alpha}|+|b_{\alpha}|\leq
\frac{4}{\pi}\inf_{\xi\in\mathbb{B}_{\ell_{p}^{n}}}\frac{1}{\left|\xi^{\alpha}\right|}
=\frac{4}{\pi}\left(\frac{|\alpha|^{|\alpha|}}{\alpha^{\alpha}}\right)^{\frac{1}{p}}.$$




Next, we prove the sharpness part. For
$z\in\mathbb{B}_{\ell_{p}^{n}}$ and some $k\in\{1,\ldots,n\}$, let
$$f(z)=\frac{2}{\pi}\arg\left(\frac{1+z_{k}^{|\alpha|}}{1-z_{k}^{|\alpha|}}\right).$$
Then
$$f(z)=\frac{2}{i\pi}
\left(\sum_{j=1}^{\infty}\frac{1}{2j-1}z_{k}^{|\alpha|(2j-1)}-\sum_{j=1}^{\infty}\frac{1}{2j-1}\overline{z}_{k}^{|\alpha|(2j-1)}\right),$$
which implies that
$$|a_{(0,\ldots,\alpha_{k},0\ldots,0)}|+|b_{(0,\ldots,\alpha_{k},0\ldots,0)}|=\frac{4}{\pi},$$
where $|\alpha|\geq1$ and $\alpha_{k}=|\alpha|.$
The proof of this theorem is finished. \qed

\subsection*{ The proof of Theorem \ref{lem-0}}
We first prove (\ref{c-1.01}).
 For $z\in\mathbb{B}_{\ell_{p}^{n}}$, let
$h(z)=\sum_{\alpha}a_{\alpha}z^{\alpha}$ and
$g(z)=\sum_{\alpha}b_{\alpha}z^{\alpha}.$ Then, by the
orthogonality, we have
$$\sum_{|\alpha|=k}a_{\alpha}z^{\alpha}=\frac{1}{2\pi}\int_{0}^{2\pi}h(ze^{i\tau})e^{-ik\tau}d\tau$$
and
$$0=\frac{1}{2\pi}\int_{0}^{2\pi}\overline{h(ze^{i\tau})}e^{-ik\tau}d\tau,$$
which give that

\begin{equation}\label{eq-c-1.11}
\sum_{|\alpha|=k}a_{\alpha}z^{\alpha}=\frac{1}{\pi}\int_{0}^{2\pi}\mbox{Re}(h(ze^{i\tau}))e^{-ik\tau}d\tau,
\end{equation} where $k\geq1$. By a similar proof process of
(\ref{eq-c-1.11}), we get

\begin{equation}\label{eq-c-1.12}
\sum_{|\alpha|=k}b_{\alpha}z^{\alpha}=\frac{1}{\pi}\int_{0}^{2\pi}\mbox{Re}(g(ze^{i\tau}))e^{-ik\tau}d\tau.
\end{equation} It follows from (\ref{eq-c-1.11}) and
(\ref{eq-c-1.12}) that

\[ -\sum_{|\alpha|=k}(a_{\alpha}+b_{\alpha})z^{\alpha}=
\frac{1}{\pi}\int_{0}^{2\pi}\big(1-\mbox{Re}(f(ze^{i\tau}))\big)e^{-ik\tau}d\tau,
\] and consequently,
\begin{eqnarray*}
\left|\sum_{|\alpha|=k}(a_{\alpha}+b_{\alpha})z^{\alpha}\right|&\leq&
\frac{1}{\pi}\int_{0}^{2\pi}\left|1-\mbox{Re}(f(ze^{i\tau}))\right|d\tau=
\frac{1}{\pi}\int_{0}^{2\pi}\left(1-\mbox{Re}(f(ze^{i\tau}))\right)d\tau\\
&=&2\left(1-\mbox{Re}(f(0))\right).
 \end{eqnarray*}

Now we prove (\ref{c-1.7}). For any
$\xi=(\xi_{1},\ldots,\xi_{n})\in\mathbb{B}_{\ell_{p}^{n}}$, let
$\xi\otimes\theta:=(\xi_{1}e^{i\theta_{1}},\ldots,\xi_{n}e^{i\theta_{n}})$,
where $\theta_{k}\in[0,2\pi]$ for all $k\in\{1,\ldots,n\}$. Then
$\xi\otimes\theta\in\mathbb{B}_{\ell_{p}^{n}}.$ It follows from the
orthogonality that

$$a_{\alpha}\xi^{\alpha}=\frac{1}{(2\pi)^{n}}\int_{0}^{2\pi}\cdots\int_{0}^{2\pi}e^{-i\sum_{k=1}^{n}\alpha_{k}
\theta_{k}}h(\xi\otimes\theta)d\theta_{1}\cdots\,d\theta_{n}$$ and
$$\frac{1}{(2\pi)^{n}}\int_{0}^{2\pi}\cdots\int_{0}^{2\pi}e^{-i\sum_{k=1}^{n}\alpha_{k}
\theta_{k}}\overline{h(\xi\otimes\theta)}d\theta_{1}\cdots\,d\theta_{n}=0,$$
which yield that

\begin{equation}\label{c-1.8}
a_{\alpha}\xi^{\alpha}=\frac{1}{(2\pi)^{n}}\int_{0}^{2\pi}\cdots\int_{0}^{2\pi}e^{-i\sum_{k=1}^{n}\alpha_{k}
\theta_{k}}
\big(h(\xi\otimes\theta)+\overline{h(\xi\otimes\theta)}\big)d\theta_{1}\cdots\,d\theta_{n}.
\end{equation}
By using a similar reasoning as in the proof of (\ref{c-1.8}), we
obtain

\begin{equation}\label{c-1.9}
b_{\alpha}\xi^{\alpha}=\frac{1}{(2\pi)^{n}}\int_{0}^{2\pi}\cdots\int_{0}^{2\pi}e^{-i\sum_{k=1}^{n}\alpha_{k}
\theta_{k}}
\big(g(\xi\otimes\theta)+\overline{g(\xi\otimes\theta)}\big)d\theta_{1}\cdots\,d\theta_{n}.
\end{equation}


We infer from (\ref{c-1.8}) and  (\ref{c-1.9}) that

$$
-(a_{\alpha}+b_{\alpha})=\frac{1}{2^{n-1}\pi^{n}}\int_{0}^{2\pi}\cdots\int_{0}^{2\pi}\frac{e^{-i\sum_{k=1}^{n}\alpha_{k}
\theta_{k}}}{\xi^{\alpha}}
\big(1-\mbox{Re}(f(\xi\otimes\theta))\big)d\theta_{1}\cdots\,d\theta_{n},
$$ and consequently,

\begin{equation}\label{c-1.10h}
|a_{\alpha}+b_{\alpha}|\leq
\frac{1}{2^{n-1}\pi^{n}}\int_{0}^{2\pi}\cdots\int_{0}^{2\pi}
\frac{\big(1-\mbox{Re}(f(\xi\otimes\theta))\big)}{\left|\xi^{\alpha}\right|}d\theta_{1}\cdots\,d\theta_{n}
=\frac{2\big(1-\mbox{Re}(f(0))\big)}{\left|\xi^{\alpha}\right|}.
\end{equation}



\noindent Hence combining (\ref{eq-c-1.27t}) and (\ref{c-1.10h})
gives the final estimate
$$|a_{\alpha}+b_{\alpha}|\leq2\big(1-\mbox{Re}(f(0))\big)\inf_{\xi\in\mathbb{B}_{\ell_{p}^{n}}}\frac{1}{\left|\xi^{\alpha}\right|}=2\big(1-\mbox{Re}(f(0))\big)
\left(\frac{|\alpha|^{|\alpha|}}{\alpha^{\alpha}}\right)^{1/p}.$$


At last, we prove the sharpness part. For
$z=(z_{1},\ldots,z_{n})\in\mathbb{B}_{\ell_{p}^{n}}$ and some
$k\in\{1,\ldots,n\}$, let
$$f(z)=\frac{-2z_{k}}{1-z_{k}}=-2\sum_{j=1}^{\infty}z_{k}^{j}
~\left(\mbox{or}~\overline{f(z)}=\frac{-2\overline{z}_{k}}{1-\overline{z}_{k}}\right).$$
 Then $f(0)=0$, $\mbox{Re}(f)<1$ and the modulus of all nonzero
 coefficients
of $f$ is  $2$, which shows that the constant $2$ in {\rm
(\ref{c-1.01})} and (\ref{c-1.7}) can not be improved. The proof of
this theorem is completed.
\qed

\begin{lem}\label{lem-3.2}
Let $\alpha=(\alpha_{1},\ldots,\alpha_{n})$ and
$m=(m_{1},\ldots,m_{n})$ be  multi-indices satisfying
$|\alpha|\geq1$ and $ m_{k}\geq\alpha_{k}$ for all
$k\in\{1,\ldots,n\}$. Then
$$\frac{|\alpha|^{|\alpha|}}{\alpha^{\alpha}}\leq n^{|m|}.$$
Furthermore, the above equality holds if and only if $\alpha=m$ and
$m_{1}=\cdots=m_{n}$.
\end{lem}

\begin{proof}
Since
 \begin{eqnarray*}
\left(\frac{|\alpha|}{\alpha_{k}}\right)^{\alpha_{k}}&=&\left(1+\frac{|\alpha|-\alpha_{k}}{\alpha_{k}}\right)^{\alpha_{k}}
\leq\left(1+\frac{|\alpha|-\alpha_{k}}{m_{k}}\right)^{m_{k}}\\&\leq&
\left(1+\frac{|m|-m_{k}}{m_{k}}\right)^{m_{k}}
=\left(\frac{|m|}{m_{k}}\right)^{m_{k}},
 \end{eqnarray*}
 in the case $\alpha_k\geq 1$,
we see that
\begin{equation}\label{gh-1}
\frac{|\alpha|^{|\alpha|}}{\alpha^{\alpha}}=\prod_{k=1}^{n}\left(\frac{|\alpha|}{\alpha_{k}}\right)^{\alpha_{k}}
\leq\prod_{k=1}^{n}\left(\frac{|m|}{m_{k}}\right)^{m_{k}}=\frac{|m|^{|m|}}{m^{m}}.
\end{equation}

Next, we show that
\[
\frac{|m|^{|m|}}{m^{m}}\leq n^{|m|}.
\]
{For any fixed $\epsilon>0$, let
$\mu(x)=(x+\epsilon)\log(x+\epsilon),~x\geq0.$ Then  $\mu$ is
strictly convex in $[0,\infty)$. It follows from Jensen's inequality
that
$$\frac{\sum_{k=1}^{n}\mu(m_{k})}{n}\geq\mu\left(\frac{\sum_{k=1}^{n}m_{k}}{n}\right).$$
Consequently,
$$\lim_{\epsilon\rightarrow0^{+}}\frac{\sum_{k=1}^{n}\mu(m_{k})}{n}\geq
\lim_{\epsilon\rightarrow0^{+}}\mu\left(\frac{\sum_{k=1}^{n}m_{k}}{n}\right),$$
which, together with (\ref{gh-1}), implies that
$$\frac{|\alpha|^{|\alpha|}}{\alpha^{\alpha}}\leq\frac{|m|^{|m|}}{m^{m}}\leq n^{|m|}.$$} The
proof of this lemma is finished.
\end{proof}

\subsection*{ The proof of Theorem \ref{thm-4}}
We first give a proof for (ii). It suffices to show for $z
\in\mathbb{B}_{\ell_{2}^{n}}\setminus \{ 0\}$. Let
$\xi=(\xi_{1},\ldots,\xi_{n})\in\mathbb{B}_{\ell_{2}^{n}}\setminus
\{ 0\}$ be fixed. Then there exists a unitary matrix $U_{\xi}$ such
that $U_{\xi }\xi^T=(\| \xi\|_2,0,\ldots, 0)^T$.
Let


\begin{equation*} P=\left(\begin{array}{cccc}
 1&0&\cdots&0\\
  0&s_{\xi}& \cdots &0\\
\vdots&& \cdots &\vdots \\
0&0& \cdots &s_{\xi}
\end{array}\right){U}_{\xi},
\end{equation*}
where $s_{\xi}=(1-\|\xi\|_{2}^{2})^{\frac{1}{2}}.$ Let
$$\varphi^{T}(z)=\frac{P(z^{T}-\xi^{T})}{1-\langle z,\xi\rangle}=\left(\frac{P_{1}(z^{T}-\xi^{T})}{1-\langle z,\xi\rangle},\ldots,
\frac{P_{n}(z^{T}-\xi^{T})}{1-\langle z,\xi\rangle}\right)^T,$$
where 
$P=(P_{1}^{T},\ldots,P_{n}^{T})^{T}$. Since
$\overline{U}_{\xi}^{T}P\xi^T=\xi^T$ and $
\overline{U}_{\xi}^{T}PQ_{\xi} =s_{\xi}Q_{\xi}, $
$\overline{U}_{\xi}^{T}\varphi^{T}(z)$ can be written as follows:
\[
\overline{U}_{\xi}^{T}\varphi^{T}(z)=-\frac{\xi^T-P_{\xi}(z)^T-s_{\xi}Q_{\xi}(z)^T}{1-\langle
z,\xi\rangle},
\]
where $P_{\xi}$ is the orthogonal projection of $\mathbb{C}^n$ onto
the subspace $[\xi]$ generated by $\xi$, and $Q_{\xi}=I-P_{\xi}$ is
the projection onto the orthogonal complement of $[\xi]$. According
to the representation of automorphism of $\mathbb{B}_{\ell_{2}^{n}}$
in \cite[Chapter 2]{R2} (or \cite{Di,LC}), we obtain that
$\varphi\in\mbox{Aut}(\mathbb{B}_{\ell_{2}^{n}})$,
 where $\mbox{Aut}(\mathbb{B}_{\ell_{2}^{n}})$ is the automorphism
group of $\mathbb{B}_{\ell_{2}^{n}}$.

Let $f$ be a pluriharmonic function of  $\mathbb{B}_{\ell_{2}^{n}}$ into $\mathbb{C}$ satisfying
$\sup_{z\in\mathbb{B}_{\ell_{2}^{n}}}|f(z)|\leq 1$. Then, by
 \cite{Vl}, we see that $f$ has a representation
$f=h+\overline{g}$, where $h$ and $g$ are holomorphic in
$\mathbb{B}_{\ell_{2}^{n}}$ with $g(0)=0$. Let
$$\mathcal{F}(z):=f(\varphi^{-1}(z))=H(z)+\overline{G(z)},$$
where
$$H(z):=h(\varphi^{-1}(z))=c_0+\sum_{k=1}^{\infty}\sum_{|\alpha|=k}c_{\alpha}z^{\alpha}$$
and
$$G(z):=g(\varphi^{-1}(z))=d_0+\sum_{k=1}^{\infty}\sum_{|\alpha|=k}d_{\alpha}z^{\alpha}.$$
Then
$$f(z)=\mathcal{F}(\varphi(z))=c_0+\sum_{k=1}^{\infty}\sum_{|\alpha|=k}c_{\alpha}u_{\alpha}(z)v_{\alpha}(z)
+\overline{d_0+\sum_{k=1}^{\infty}\sum_{|\alpha|=k}d_{\alpha}u_{\alpha}(z)v_{\alpha}(z)},$$
where
$u_{\alpha}(z)=\prod_{j=1}^{n}(P_{j}(z^{T}-\xi^{T}))^{\alpha_{j}}$
and $v_{\alpha}(z)=(1-\langle z,\xi\rangle)^{-|\alpha|}$. Then for
the multi-index $m=(m_{1},\ldots,m_{n})\neq 0$, we have

\begin{equation}\label{v-1}\frac{\partial^{|m|} f(z)}{\partial
z_{1}^{m_{1}}\cdots
\partial z_{n}^{m_{n}}}=\sum_{k=1}^{\infty}\sum_{|\alpha|=k}c_{\alpha}\frac{\partial^{|m|} (u_{\alpha}(z)v_{\alpha}(z))}{\partial
z_{1}^{m_{1}}\cdots
\partial z_{n}^{m_{n}}}
\end{equation}
and
\begin{equation}\label{v-2}
\frac{\partial^{|m|} f(z)}{\partial \overline{z}_{1}^{m_{1}}\cdots
\partial \overline{z}_{n}^{m_{n}}}=\overline{\sum_{k=1}^{\infty}\sum_{|\alpha|=k}d_{\alpha}\frac{\partial^{|m|} (u_{\alpha}(z)v_{\alpha}(z))}{\partial
z_{1}^{m_{1}}\cdots
\partial z_{n}^{m_{n}}}}.
\end{equation}
By elementary calculations, we see that for $|m|\geq|\alpha|$,

\begin{equation}\label{v-3}
\frac{\partial^{|m|} (u_{\alpha}(\xi)v_{\alpha}(\xi))}{\partial
z_{1}^{m_{1}}\cdots
\partial z_{n}^{m_{n}}}=\sum_{|\beta|=|\alpha|, m_j\geq \beta_j}\frac{\partial^{|\beta|} u_{\alpha}(\xi)}{\partial
z_{1}^{\beta_{1}}\cdots
\partial z_{n}^{\beta_{n}}}\frac{\partial^{|m|-|\beta|} v_{\alpha}(\xi)}{\partial
z_{1}^{m_{1}-\beta_{1}}\cdots
\partial z_{n}^{m_{n}-\beta_{n}}}\prod_{j=1}^{n}{m_{j}\choose \beta_{j}},
\end{equation}
where $\beta=(\beta_{1},\ldots,\beta_{n})$ is a multi-index.

$\mathbf{Claim~ 1.}$
\[ \left|\frac{\partial^{|\beta|}
u_{\alpha}(\xi)}{\partial z_{1}^{\beta_{1}}\cdots
\partial z_{n}^{\beta_{n}}}\right|
\leq |\alpha|!
\]
for all $\beta$ with $|\beta|=|\alpha|$.

Now we prove Claim 1. Set
$$P=(p_{jk})_{1\leq j,k\leq n}.$$
Then, we have
\[
u_{\alpha}(z)=
\prod_{k=1}^{n}\left(\sum_{j=1}^{n}p_{kj}(z_{j}-\xi_{j})\right)^{\alpha_{k}}.
\]
By calculating the partial derivative directly from the above
formula and using $|p_{jk}|\leq1~(1\leq j,k\leq n)$, we can prove
that
\[ \left|\frac{\partial^{|\beta|}
u_{\alpha}(\xi)}{\partial z_{1}^{\beta_{1}}\cdots
\partial z_{n}^{\beta_{n}}}\right|
\leq \left|\frac{\partial^{|\beta|}
\tilde{u}_{\alpha}(\xi)}{\partial z_{1}^{\beta_{1}}\cdots
\partial z_{n}^{\beta_{n}}}\right|,
\]
where
\[
\tilde{u}_{\alpha}(z)=
\prod_{k=1}^{n}\left(\sum_{j=1}^{n}(z_{j}-\xi_{j})\right)^{\alpha_{k}}
=\left(\sum_{j=1}^{n}(z_{j}-\xi_{j})\right)^{|\alpha|}.
\]
Since
\[
\frac{\partial^{|\beta|} \tilde{u}_{\alpha}(\xi)}{\partial
z_{1}^{\beta_{1}}\cdots
\partial z_{n}^{\beta_{n}}}=|\alpha|!,
\]
we obtain the desired result.

$\mathbf{Claim~ 2.}$ For $|\beta|=|\alpha|\geq 1$ and $m_j \geq
\beta_j$ for $j=1,\dots, n$,  \[ \left|\frac{\partial^{|m|-|\beta|}
v_{\alpha}(\xi)}{\partial z_{1}^{m_{1}-\beta_{1}}\cdots
\partial z_{n}^{m_{n}-\beta_{n}}}\right|\leq\frac{(|m|-1)!}{(|\alpha|-1)!}
\frac{\prod_{j=1}^{n}
|\xi_{j}|^{m_{j}-\beta_{j}}}{(1-\|\xi\|_{2}^{2})^{|m|}}.\]

Next, we prove Claim 2. Elementary computations show that
\[ \frac{\partial^{|m|-|\beta|} v_{\alpha}(z)}{\partial
z_{1}^{m_{1}-\beta_{1}}\cdots
\partial z_{n}^{m_{n}-\beta_{n}}}=|\alpha|(|\alpha|+1)\cdots(|m|-1)\frac{\prod_{j=1}^{n}
\overline{\xi}_{j}^{m_{j}-\beta_{j}}}{(1-\langle
z,\xi\rangle)^{|m|}}.\] Then replacing $z$ by $\xi$, implies that
Claim 2 is true. The proof of this claim is finished.

Let $$\partial^{m}\omega_{\alpha}(\xi):=\frac{\partial^{|m|}
\omega_{\alpha}(\xi)}{\partial z_{1}^{m_{1}}\cdots
\partial z_{n}^{m_{n}}},$$ where $\omega_{\alpha}=u_{\alpha}v_{\alpha}$. Then  combining
(\ref{v-3}), Claims 1 and 2 gives

 \begin{eqnarray*}\label{v-i4} \left|\partial^{m}\omega_{\alpha}(\xi)\right|&\leq&\sum_{|\beta|=|\alpha|, m_j\geq \beta_j}|\alpha|!
\frac{(|m|-1)!}{(|\alpha|-1)!(1-\|\xi\|_{2}^{2})^{|m|}}\prod_{j=1}^{n}|\xi_j|^{m_j-\beta_j}{m_{j}\choose
\beta_{j}}\\
&\leq& \sum_{|\beta|=|\alpha|, m_j\geq \beta_j}
\frac{|m|!}{(1-\|\xi\|_{2}^{2})^{|m|}}\prod_{j=1}^{n}|\xi_j|^{m_j-\beta_j}{m_{j}\choose
\beta_{j}},
\end{eqnarray*}
which, together with (\ref{v-1}), (\ref{v-2}),  Lemmas \ref{lem-3.1} and \ref{lem-3.2}, implies that
\beq\label{long-1}
&&\left|\frac{\partial^{|m|} f(\xi)}{\partial
z_{1}^{m_{1}}\cdots
\partial z_{n}^{m_{n}}}\right|+\left|\frac{\partial^{|m|} f(\xi)}{\partial
\overline{z}_{1}^{m_{1}}\cdots
\partial \overline{z}_{n}^{m_{n}}}\right|
\leq
\sum_{k=1}^{\infty}\sum_{|\alpha|=k}\big(|c_{\alpha}|+|d_{\alpha}|)\left|\partial^{m}\omega_{\alpha}(\xi)\right|\\ \nonumber
&\leq&\sum_{k=1}^{|m|}\sum_{|\alpha|=k}\frac{4}{\pi}\left(\frac{|\alpha|^{|\alpha|}}{\alpha^{\alpha}}\right)^{\frac{1}{2}}
\sum_{|\beta|=|\alpha|, m_j\geq \beta_j}
\frac{|m|!}{(1-\|\xi\|_{2}^{2})^{|m|}}\prod_{j=1}^{n}|\xi_j|^{m_j-\beta_j}{m_{j}\choose
\beta_{j}}\\ \nonumber
&\leq&\frac{4n^{\frac{|m|}{2}}}{\pi(1-\|\xi\|_{2}^{2})^{|m|}}
\Psi(\xi),
\eeq
where $$\Psi(\xi)=\sum_{k=1}^{|m|}\sum_{|\alpha|=k} \sum_{|\beta|=|\alpha|, m_j\geq
\beta_j} |m|!\prod_{j=1}^{n}|\xi_j|^{m_j-\beta_j}{m_{j}\choose
\beta_{j}}.$$
In the following, we begin to estimate $\Psi(\xi).$
Since the dimension of the space of $k$-homogeneous polynomials in $\mathbb{C}^n$ is ${n+k-1\choose
n-1}$,  we see that

\beq\label{long-2} \Psi(\xi)
&=&|m|!
\sum_{k=1}^{|m|} \sum_{|\beta|=k, m_j\geq
\beta_j}\prod_{j=1}^{n}|\xi_j|^{m_j-\beta_j}{m_{j}\choose
\beta_{j}}{n+k-1\choose
n-1}\\ \nonumber
&\leq&|m|!
\prod_{j=1}^{n}(1+|\xi_j|)^{m_j} {n+|m|-1\choose n-1}.
 \eeq

  Therefore, substituting (\ref{long-2}) into (\ref{long-1}) and replacing $\xi$ by $z$, we can get the desired
 result.

Next, we give a proof for (i). As in the proof for (ii), we have
\begin{eqnarray*} \left|\frac{\partial^{m} f(\xi)}{\partial
z^{m}}\right|+\left|\frac{\partial^{m} f(\xi)}{\partial
\overline{z}^{m}}\right| &\leq&
\sum_{\alpha=1}^{\infty}\big(|c_{\alpha}|+|d_{\alpha}|)\left|\partial^{m}\omega_{\alpha}(\xi)\right|\\
&\leq& \frac{4}{\pi(1-|\xi|^{2})^{m}} \sum_{\alpha=1}^{m} \alpha
(m-1)!|\xi|^{m-\alpha}{m\choose
\alpha}\\
&=& \frac{4}{\pi(1-|\xi|^{2})^{m}} m!(1+|\xi|)^{m-1}.
 \end{eqnarray*}

{Now we prove the sharpness part for $n=1$. For $z\in\mathbb{D}$,
 let
$$f_{m}(z)=\frac{2}{\pi}\arg\left(\frac{1+z^{m}}{1-z^{m}}\right)
=\frac{2}{i\pi} \left(\sum_{j=1}^{\infty} \frac{1}{2j-1}z^{(2j-1)m}-
\sum_{j=1}^{\infty}\frac{1}{2j-1}\overline{z}^{(2j-1)m}\right).$$
Then $f_m$ is harmonic on $\mathbb{D}$, $|f_m(z)|<1$ for $z\in
\mathbb{D}$ and
\[\left|\frac{\partial^{m} f_{m}(0)}{\partial
z^{m}}\right|+\left|\frac{\partial^{m} f_{m}(0)}{\partial
\overline{z}^{m}}\right|=\frac{4}{\pi}m!.\]} The proof of this
theorem is completed. \qed


\subsection*{ The proof of Theorem \ref{thm-2}}
From \cite{Vl}, we know that
$f\in\mathscr{PH}(\mathbb{B}_{\ell_{\infty}^{n}})$ has a
representation $f=h+\overline{g}$, where $h$ and $g$ are holomorphic
in $\mathbb{B}_{\ell_{\infty}^{n}}$ with $g(0)=0$. Let
$\alpha=(\alpha_{1},\ldots,\alpha_{n})$ be a multi-index. Then $f$
can be expressed as a power series in
$\mathbb{B}_{\ell_{\infty}^{n}}$ as follows:

$$f(z)=h(z)+\overline{g(z)}=\sum_{\alpha}a_{\alpha}z^{\alpha}+\sum_{\alpha}\overline{b}_{\alpha}\overline{z}^{\alpha}.$$
It follows from (\ref{c-1.7}) that, for all $|\alpha|\geq1$,

\begin{equation}\label{c-1.1}
 |a_{\alpha}+b_{\alpha}|\leq2\big(1-{\rm
Re}(f(0))\big). \end{equation}

For $k\in\{1,\ldots,n\}$ and
$z=(z_{1},\ldots,z_{n})\in\mathbb{B}_{\ell_{\infty}^{n}},$ let
$$\phi(\zeta)=(\phi_{1}(\zeta_{1}),\ldots,\phi_{n}(\zeta_{n})),$$
where
$\zeta=(\zeta_{1},\ldots,\zeta_{n})\in\mathbb{B}_{\ell_{\infty}^{n}}$
and
$$\phi_{k}(\zeta_{k})=\frac{z_{k}+\zeta_{k}}{1+\overline{z_{k}}\zeta_{k}}.$$
Then $f\circ\phi$ can be written in the following form:
$$\chi(\zeta)=f(\phi(\zeta))=H(\zeta)+\overline{G(\zeta)}=\sum_{\alpha}c_{\alpha}\zeta^{\alpha}+\sum_{\alpha}\overline{d}_{\alpha}\overline{\zeta}^{\alpha},$$
where $H=h\circ\phi$ and  $G=g\circ\phi$. By (\ref{c-1.1}), we see
that

\begin{equation}\label{c-1.4}|c_{\alpha}+d_{\alpha}|\leq
2\big(1-\mbox{Re}(\chi(0))\big)=2\big(1-\mbox{Re}(f(z))\big).
\end{equation}

Let $z\in\mathbb{B}_{\ell_{\infty}^{n}}$. We may assume that
$m_k\geq 1$ for $k\in\{1,\ldots,n\}$. Then as in the proof of
\cite[Theorem 1]{CR}, by using the Cauchy integral formula (cf.
\cite{R1}) and by changing the integral variable
$\eta_{k}=\phi_{k}(\zeta_{k})=\frac{z_{k}+\zeta_{k}}{1+\overline{z}_{k}\zeta_{k}}$,
we have


\begin{equation}\label{ch-pp-1}
\begin{split} &\left|\frac{\partial^{|m|} f(z)}{\partial
z_{1}^{m_{1}}\cdots
\partial z_{n}^{m_{n}}}+ \overline{\frac{\partial^{|m|} f(z)}{\partial
\overline{z}_{1}^{m_{1}}\cdots
\partial \overline{z}_{n}^{m_{n}}}}\right|\\
&\leq\frac{m!}{\prod_{k=1}^{n}(1-|z_{k}|^{2})^{m_{k}}}
\sum_{j_{1}=0}^{m_{1}-1}\cdots\sum_{j_{n}=0}^{m_{n}-1}{m_{1}-1\choose
j_{1}}\cdots{m_{n}-1\choose j_{n}}\\
&\times
\big(|c_{m_{1}-j_{1},\cdots,m_{n}-j_{n}}+d_{m_{1}-j_{1},\cdots,m_{n}-j_{n}}|\big)\prod_{k=1}^{n}|z_{k}|^{j_{k}}.
\end{split}
\end{equation}
Combining (\ref{c-1.4}) and (\ref{ch-pp-1}),
 we conclude that

\begin{eqnarray*}
 &&\left|\frac{\partial^{|m|} f(z)}{\partial
z_{1}^{m_{1}}\cdots
\partial z_{n}^{m_{n}}}+ \overline{\frac{\partial^{|m|} f(z)}{\partial
\overline{z}_{1}^{m_{1}}\cdots
\partial \overline{z}_{n}^{m_{n}}}}\right|\\
&\leq&2\big(1-\mbox{Re}(f(z))\big)\frac{m!}{\prod_{k=1}^{n}(1-|z_{k}|^{2})^{m_{k}}}
\sum_{j_{1}=0}^{m_{1}-1}\cdots\sum_{j_{n}=0}^{m_{n}-1}{m_{1}-1\choose
j_{1}}\cdots{m_{n}-1\choose
j_{n}}\prod_{k=1}^{n}|z_{k}|^{j_{k}}\\
&\leq&2\big(1-\mbox{Re}(f(z))\big)\frac{m!}{\prod_{k=1}^{n}(1-|z_{k}|^{2})^{m_{k}}}\prod_{k=1}^{n}(1+|z_{k}|)^{m_{k}-1}\\
&=&m!2\big(1-\mbox{Re}(f(z))\big)\prod_{k=1}^{n}\frac{(1+|z_{k}|)^{m_{k}-1}}{(1-|z_{k}|^{2})^{m_{k}}}
\leq
m!2\big(1-\mbox{Re}(f(z))\big)\frac{(1+\|z\|_{\infty})^{|m|-n}}{(1-\|z\|_{\infty}^{2})^{|m|}}.
\end{eqnarray*}
Now we show the sharpness part. For
$z=(z_{1},\ldots,z_{n})\in\mathbb{B}_{\ell_{\infty}^{n}}$ and some
$k\in\{1,\ldots,n\}$, let
$$f(z)=\frac{-2z_{k}}{1-z_{k}}=-2\sum_{j=1}^{\infty}z_{k}^{j}~\left(\mbox{or}~\overline{f(z)}=\frac{-2\overline{z}_{k}}{1-\overline{z}_{k}}\right).$$
 Then $f(0)=0$, $\mbox{Re}(f)<1$ and

 $$\left|\frac{\partial^{m_{k}} f(0)}{\partial z_{k}^{m_{k}}}\right|=2(m_{k}!)$$
which implies that the constant $2$ in {\rm (\ref{c-1.5})} is sharp.
The proof of this theorem is completed. \qed

\begin{lem}\label{lem-3.1ch}
For $p\in[1,\infty]$, let
$f(z)=\sum_{\alpha}a_{\alpha}z^{\alpha}+\sum_{\alpha}\overline{b}_{\alpha}\overline{z}^{\alpha}$ be a pluriharmonic function of $\mathbb{B}_{\ell_{p}^{n}}$
into $\mathbb{D}$.
 Then  $k\geq1$,
 \[
 \sup_{z\in\mathbb{B}_{\ell_{p}^{n}}}\left(\sum_{|\alpha|=k}(|a_{\alpha}|^{2}+|b_{\alpha}|^{2})|z^{\alpha}|^{2}\right)\leq\frac{16}{\pi^{2}}.
 \]
\end{lem}

\bpf  For
$z\in\mathbb{B}_{\ell_{p}^{n}},$ $\vartheta\in[0,2\pi]$ and
$k\geq1$, it follows from the orthogonality that
$$\sum_{|\alpha|=k}a_{\alpha}z^{\alpha}=\frac{1}{2\pi}\int_{0}^{2\pi}f(e^{i\vartheta}z)e^{-ik\vartheta}\,d\vartheta$$
and
$$\sum_{|\alpha|=k}\overline{b}_{\alpha}\overline{z}^{\alpha}=\frac{1}{2\pi}\int_{0}^{2\pi}f(e^{i\vartheta}z)e^{ik\vartheta}\,d\vartheta,$$
which, together with  Lemma \ref{lem-CR}, implies that

\begin{equation}\label{jp-1-1}
\begin{split}
\left|\sum_{|\alpha|=k}a_{\alpha}z^{\alpha}+\sum_{|\alpha|=k}\overline{b}_{\alpha}\overline{z}^{\alpha}\right|
&\leq\frac{1}{2\pi}
\int_{0}^{2\pi}\left|f(e^{i\vartheta}z)\right|\left|e^{-ik\vartheta}+e^{ik\vartheta}\right|\,d\vartheta\\
 &\leq\frac{1}{\pi}\int_{0}^{2\pi}|\cos
k\vartheta|\,d\vartheta =\frac{4}{\pi}.
\end{split}
 \end{equation}
 Let
$z=(r_{1}e^{i\vartheta_{1}},\ldots,r_{n}e^{i\vartheta_{n}})\in\mathbb{B}_{\ell_{p}^{n}}$,
where $\vartheta_{j}\in[0,2\pi]$ and $r_{j}\geq0$ for all
$j\in\{1,\ldots,n\}$. Then, by (\ref{jp-1-1}), we have

\begin{eqnarray*}
\frac{16}{\pi^{2}}&\geq&\frac{1}{(2\pi)^{n}}\int_{0}^{2\pi}\cdots\int_{0}^{2\pi}\left|\sum_{|\alpha|=k}a_{\alpha}z^{\alpha}+
\sum_{|\alpha|=k}\overline{b}_{\alpha}\overline{z}^{\alpha}\right|^{2}d\vartheta_{1}\cdots\,d\vartheta_{n}\\
&=&\sum_{|\alpha|=k}(|a_{\alpha}|^{2}+|b_{\alpha}|^{2})|z^{\alpha}|^{2},
\end{eqnarray*}  which completes the proof.
\epf

\subsection*{ The proof of Theorem \ref{thm-2.1x}}
 We first  prove the left hand side of the
inequality (\ref{eq-T-1}). Let
$f(z)=\sum_{\alpha}a_{\alpha}z^{\alpha}+\sum_{\alpha}\overline{b}_{\alpha}\overline{z}^{\alpha}\in\mathscr{PH}(\mathbb{B}_{\ell_{p}^{n}})$
with $b_0=0$ and $\sup_{z\in\mathbb{B}_{\ell_{p}^{n}}}|f(z)|\leq 1$.
We split the remaining proof into two cases.

\bca
$p\in[2,\infty]$.
\eca

For $\|z\|_{p}\leq \rho_{n}:=\pi/((\pi+4\sqrt{2})\sqrt{n})$ and
$\zeta=z/\rho_{n}$, it follows from $\sum_{|\alpha|=k}1\leq n^{k}$
and the Cauchy-Schwarz inequality that

\begin{eqnarray*} \sum_{|\alpha|=k}(|a_{\alpha}|+|b_{\alpha}|)|z^{\alpha}|&\leq&
\left(\sum_{|\alpha|=k}(|a_{\alpha}|+|b_{\alpha}|)^{2}|\zeta^{\alpha}|^{2}\right)^{1/2}\left(\sum_{|\alpha|=k}\rho_{n}^{2k}\right)^{1/2}
\\&\leq&\sqrt{2}\left(\sum_{|\alpha|=k}(|a_{\alpha}|^{2}+|b_{\alpha}|^{2})\left|\zeta^{\alpha}\right|^{2}\right)^{1/2}n^{\frac{k}{2}}\rho_{n}^{k},
\end{eqnarray*} which, together with
Lemma \ref{lem-3.1ch}, yields that

\begin{equation}\label{f-b-1}\sum_{|\alpha|=k}(|a_{\alpha}|+|b_{\alpha}|)|z^{\alpha}|\leq\frac{4\sqrt{2}}{\pi}n^{\frac{k}{2}}\rho_{n}^{k}=\frac{4\sqrt{2}}{\pi}\frac{1}{\left(1+\frac{4\sqrt{2}}{\pi}\right)^{k}}.
\end{equation} Consequently, by (\ref{f-b-1}), we have

 \[ \sum_{k=1}^{\infty}\sum_{|\alpha|=k}(|a_{\alpha}|+|b_{\alpha}|)|z^{\alpha}|\leq\frac{4\sqrt{2}}{\pi}\sum_{k=1}^{\infty}\frac{1}{\left(1+\frac{4\sqrt{2}}{\pi}\right)^{k}}
 =1.
 \]

 Therefore,
$\mathcal{R}^{\ast}_{P}(\mathbb{B}_{\ell_{p}^{n}})\geq\pi/((\pi+4\sqrt{2})\sqrt{n}).$

\bca
$p\in[1,2)$.
\eca

By Theorem \ref{lem-3.1}, we have

\begin{equation}\label{eq-yhh-1}
\begin{split}
\sum_{|\alpha|=k}(|a_{\alpha}|+|b_{\alpha}|)|z^{\alpha}|&\leq\frac{4}{\pi}\sum_{|\alpha|=k}\left(\frac{|\alpha|^{|\alpha|}}{\alpha^{\alpha}}\right)^{\frac{1}{p}}|z^{\alpha}|
\leq
\frac{4}{\pi}\sum_{|\alpha|=k}\frac{|\alpha|^{|\alpha|}}{\alpha^{\alpha}}|z^{\alpha}|\\
&=\frac{4}{\pi}\sum_{|\alpha|=k}\frac{k^{k}}{\alpha^{\alpha}\cdot\frac{k!}{\alpha!}}\cdot\frac{k!}{\alpha!}|z^{\alpha}|.
\end{split}
\end{equation}

By applying H\"older's inequality in the case $p\in (1,2)$, we see
that
\begin{equation}\label{uk-1}\sum_{|\alpha|=k}\frac{k!}{\alpha!}|z^{\alpha}|=\|z\|_{1}^{k}\leq
n^{k\left(1-\frac{1}{p}\right)}\|z\|_{p}^{k}.
\end{equation}
Combining (\ref{eq-yhh-1}) and (\ref{uk-1}), and using
$\alpha^{\alpha}/(\alpha!)\geq1$, we have

\[\sum_{k=1}^{\infty}\sum_{|\alpha|=k}(|a_{\alpha}|+|b_{\alpha}|)|z^{\alpha}|\leq\frac{4}{\pi}\sum_{k=1}^{\infty}
\frac{k^{k}}{k!}\left(n^{1-\frac{1}{p}}\|z\|_{p}\right)^{k}.\]
 Hence
$\mathcal{R}^{\ast}_{P}(\mathbb{B}_{\ell_{p}^{n}})\geq
x_{0}/n^{1-1/p}$, where $x_{0}$ is the unique positive solution to
the following equation

\begin{equation}\label{eq-yhh-3}
\sum_{k=1}^{\infty}\frac{k^{k}}{k!}x^{k}=\frac{\pi}{4}.
\end{equation}
By \cite[p. 328]{B-2000}, we have
\[\sum_{k=1}^{\infty}\frac{k^{k}}{k!}\left(\frac{1}{3\sqrt[3]{e}}\right)^{k}=\frac{1}{2},\]
which implies that the unique positive solution $x_{0}$ to the
equation (\ref{eq-yhh-3}) is  bigger than $1/(3\sqrt[3]{e})$.

Next, we  prove the right hand side of the
inequality (\ref{eq-T-1}).
By
the Kahane-Salem-Zygmund inequality (see \cite[Corollary]{B-2000} or
\cite{BK}), for $n\geq 2$ and $k\geq 2$, there exist coefficients
$(c_{\alpha})_{|\alpha|=k}$ with $|c_{\alpha}|=k!/\alpha!$ for all
$\alpha$ such that

\[
\sup_{z\in\mathbb{B}_{\ell_{p}^{n}}}\left|\sum_{|\alpha|=k}c_{\alpha}z^{\alpha}\right|\leq\sqrt{32k\log
(6k)}n^{\frac{1}{2}+\left(\frac{1}{2}-\frac{1}{\max\{2,p\}}\right)k}(k!)^{1-\frac{1}{\min\{p,2\}}}.
\]
Since
\[\label{rrt-1}\sum_{|\alpha|=k}\frac{k!}{\alpha!}=n^{k},
\] we
see from the definition of
$\mathcal{R}^{\ast}_{P}(\mathbb{B}_{\ell_{p}^{n}})$ that
\begin{eqnarray*}\left(\frac{\mathcal{R}^{\ast}_{P}(\mathbb{B}_{\ell_{p}^{n}})}{n^{1/p}}\right)^{k}n^{k}
&=&
\sum_{|\alpha|=k}|c_{\alpha}|\left(\frac{\mathcal{R}^{\ast}_{P}(\mathbb{B}_{\ell_{p}^{n}})}{n^{1/p}}\right)^{k}
\\
&\leq& \sqrt{32k\log
(6k)}n^{\frac{1}{2}+\left(\frac{1}{2}-\frac{1}{\max\{2,p\}}\right)k}(k!)^{1-\frac{1}{\min\{p,2\}}}.
\end{eqnarray*} Consequently,
\[\mathcal{R}^{\ast}_{P}(\mathbb{B}_{\ell_{p}^{n}})\leq\left(\frac{(k!)^{\frac{1}{k}}}{n}\right)^{1-\frac{1}{\min\{p,2\}}}\big(32kn\log
(6k)\big)^{\frac{1}{2k}}.\] Choosing $k$ to be an integer close to
$\log n$, we  get the desired result.
The
proof of this theorem is completed.
\qed

Let $$P(z)=\sum_{|\alpha|=k}a_{\alpha}z^{\alpha}$$ be a
$k$-homogeneous polynomial in $\mathbb{B}_{\ell_{\infty}^{n}}$.
Using polarization, Bohnenblust and Hille
\cite{BH-1931} obtained the following inequality for $k$-homogeneous
polynomials on $(\mathbb{C}^{n},
\|\cdot\|_{\infty})$: for any $k\geq1$, there exists a
constant $D_{k}\geq1$ such that, for any complex $k$-homogeneous
polynomial $P(z)=\sum_{|\alpha|=k}a_{\alpha}z^{\alpha}$ on
$(\mathbb{C}^{n},
\|\cdot\|_{\infty})$, we have
\begin{equation}\label{e-BH}
\left(\sum_{|\alpha|=k}|a_{\alpha}|^{\frac{2k}{k+1}}\right)^{\frac{k+1}{2k}}\leq
D_{k}\|P\|_{\infty},
\end{equation} where
$\|P\|_{\infty}=\sup_{z\in\mathbb{B}_{\ell_{\infty}^{n}}}|P(z)|$. In
the following, the best constant $D_{k}$ in (\ref{e-BH}) will be
denoted by ${\rm B^{pol}_{\mathbb{C},k}}$ (see \cite{BPS}).
Recently,  Bayart, Pellegrino and Seoane-Sep\'ulveda  (\cite[Theorem
1.1]{BPS}) obtained the following estimate: For
any $\varepsilon>0$, there exists $\kappa>0$ such that, for any
$k\geq1$,

\begin{equation}\label{eq-c-1.14}
{\rm
B^{pol}_{\mathbb{C},k}}\leq\kappa(1+\varepsilon)^{k}.
\end{equation}


A Schauder basis $(x_{n})$ of a Banach space $X$ is said to be
unconditional if there is a constant $C\geq0$ such that
$\|\sum_{k=1}^{n}\epsilon_{k}\beta_{k}x_{k}\|\leq
C\|\sum_{k=1}^{n}\beta_{k}x_{k}\|$ for all $n$ and
$\beta_{1},~\ldots,~\beta_{n},$
$\epsilon_{1},~\ldots,~\epsilon_{n}\in\mathbb{C}$ with
$|\epsilon_{k}|\leq1$. In particular, the best constant $C$ is
called the unconditional basis constant of $(x_{n})$. If
$X=(\mathbb{C}^{n},\|\cdot\|)$ is a Banach space and
$k\in\{1,2,\ldots\}$, then $\mathscr{P}(^{k}X)$ stands for the
Banach space of all $k$-homogeneous polynomials
$P(z)=\sum_{|\alpha|=k}c_{\alpha}z^{\alpha},$ $z\in\mathbb{C}^{n}$,
together with the norm
$\|P\|_{\mathscr{P}(^{k}X)}:=\sup_{\|z\|\leq1}|P(z)|$. The
unconditional basis constant of all monomials $z^{\alpha}$ with
$|\alpha|=k$ is denoted by $\chi_{{\rm mon}}(\mathscr{P}(^{k}X))$.
For more details on this topic, we refer to \cite{DF-11}.

\subsection*{ The proof of Theorem \ref{thm-2.0x}} We divide the proof of
this theorem into three cases.

\bca $n=1$.
\eca

 Let
$f(z)=\sum_{j=0}^{\infty}a_{j}z^j+\sum_{j=0}^{\infty}\overline{b}_{j}
\overline{z}^j \in\mathscr{PH}_{+}(\mathbb{B}_{\ell_{p}^{1}})$ with
$b_0=0$. For $z\in\mathbb{B}_{\ell_{p}^{1}}$, by (\ref{c-1.01}), we
have
\[\sum_{j=0}^{\infty}|(a_{j}+b_{j})z^{j}|=
f(0)+\sum_{j=1}^{\infty}|(a_{j}+b_{j})z^{j}|\leq
f(0)+2(1-f(0))\frac{|z|}{1-|z|},\] and consequently,

\[\sum_{j=0}^{\infty}\left|(a_{j}+b_{j})\frac{1}{3^{j}}\right|\leq1.
\]

Next, we prove the sharpness part. For
$z\in\mathbb{B}_{\ell_{p}^{1}}$, let
$$f(z)=\frac{-2z}{1-z}=-2\sum_{k=1}^{\infty}z^{k}~\left(\mbox{or}~\overline{f(z)}=\frac{-2\overline{z}}{1-\overline{z}}\right).$$
It is not difficult to know that $\mbox{Re}(f)\leq1$ and
$$M(|z|):=\sum_{k=1}^{\infty}2|z|^{k}=\frac{2|z|}{1-|z|},$$
which gives that $M(\frac{1}{3})=1$.

\bca
$n\geq2~\mbox{and}~p\in[1,\infty)$.
\eca

\bst We first prove that there is an absolute
constant $C$ such that
$$\mathcal{R}_{P}(\mathbb{B}_{\ell_{p}^{n}})\leq C \left(\frac{\log
n}{n}\right)^{1-\frac{1}{\min\{p,2\}}}.$$
\est

This step easily follows from the proof of Theorem \ref{thm-2.1x}
by replacing 
$\mathcal{R}^{\ast}_{P}(\mathbb{B}_{\ell_{p}^{n}})$
by
$\mathcal{R}_{P}(\mathbb{B}_{\ell_{p}^{n}})$.





\bst
Next, we prove that there is an absolute
constant $C$ such that
$$\mathcal{R}_{P}(\mathbb{B}_{\ell_{p}^{n}})\geq \frac{1}{C} \left(\frac{\log n}{n}\right)^{1-\frac{1}{\min\{p,2\}}}.$$
\est

Let $f(z)=\sum_{\alpha}a_{\alpha}z^{\alpha}
+\sum_{\alpha}\overline{b}_{\alpha}\overline{z}^{\alpha}
\in\mathscr{PH}_{+}(\mathbb{B}_{\ell_{p}^{n}})$ with $b_0=0$. By
Theorem \ref{lem-0}, for all $k\in\{1,2,\ldots\}$, we have
\[
\sum_{|\alpha|=k}|a_{\alpha}+b_{\alpha}|\left|z^{\alpha}\right| \leq
2(1-{\rm Re} (f(0))) \chi_{{\rm
mon}}(\mathcal{P}(^k\ell_p^n))\|z\|_p^k.
\]
Since there exists an absolute constant $C>0$ such that (see
\cite[p.144]{DF-11}) \[ \chi_{{\rm
mon}}(\mathscr{P}(^{k}\ell_{p}^{n}))\leq
C^{k}\left(1+\frac{n}{k}\right)^{(k-1)\left(1-\frac{1}{\min\{p,2\}}\right)},\]
we obtain
\[
\sum_{\alpha} |a_{\alpha}+b_{\alpha}|\left|z^{\alpha}\right|\leq
f(0)+2\big(1-f(0)\big)\sum_{k=1}^{\infty}C^k
\left(1+\frac{n}{k}\right)^{(k-1)\left(1-\frac{1}{\min\{p,2\}}\right)}\|z\|_p^k.
\]
At last, by using a similar proof process of \cite[Theorem
1.1]{DF-11}, we obtain the desired result.

\bca
$n\geq2$ and $p=\infty$.
\eca

We divide the proof of
this situation into two parts (see Claims 1 and 2).

$\mathbf{Claim~ 1.}$
$\limsup_{n\rightarrow\infty}\frac{\mathcal{R}_{P}(\mathbb{B}_{\ell_{\infty}^{n}})}{\sqrt{\frac{\log
n}{n}}}\geq1$.

 Let
$f(z)=\sum_{\alpha}a_{\alpha}z^{\alpha}
+\sum_{\alpha}\overline{b}_{\alpha}\overline{z}^{\alpha}
\in\mathscr{PH}_{+}(\mathbb{B}_{\ell_{\infty}^{n}})$ with $b_0=0$.
Elementary computations show that, for
$z\in\mathbb{B}_{\ell_{\infty}^{n}}$,

\begin{equation}\label{c-1.2}
\sum_{\alpha}|(a_{\alpha}+b_{\alpha})z^{\alpha}|\leq
f(0)+\sum_{k=1}^{\infty}\|z\|_{\infty}^{k}\sum_{|\alpha|=k}|a_{\alpha}+b_{\alpha}|.
\end{equation} By (\ref{eq-c-1.14}) and
(\ref{c-1.01}), for any $\varepsilon>0$, there exists $\kappa>0$
such that, for any $k\geq1$,

\begin{equation}\label{eq-c-1.18}
\begin{split}
\left(\sum_{|\alpha|=k}\left(|a_{\alpha}+b_{\alpha}|^{\frac{2k}{k+1}}\right)\right)^{\frac{k+1}{2k}}
&\leq\kappa(1+\varepsilon)^{k}\sup_{z\in\mathbb{B}_{\ell_{\infty}^{n}}}\left|\sum_{|\alpha|=k}(a_{\alpha}+b_{\alpha})z^{\alpha}\right|
\\
&\leq 2\big(1-f(0)\big)\kappa(1+\varepsilon)^{k}.
\end{split}
\end{equation}

Since the dimension of the space of $k$-homogeneous polynomials in
$\mathbb{C}^{n}$ is ${n+k-1\choose k}$, an application of
(\ref{c-1.01}), (\ref{eq-c-1.18}) and H\"older's inequality in the
case $k\geq 2$ to the sum $\sum_{|\alpha|=k}|a_{\alpha}+b_{\alpha}|$
gives

\begin{equation}\label{c-1.3}
\begin{split}
\sum_{|\alpha|=k}|a_{\alpha}+b_{\alpha}| &\leq{n+k-1\choose
k}^{\frac{k-1}{2k}}\left(\sum_{|\alpha|=k}\left(|a_{\alpha}+b_{\alpha}|^{\frac{2k}{k+1}}\right)\right)^{\frac{k+1}{2k}}
\\
&\leq2\big(1-f(0)\big)\kappa(1+\varepsilon)^{k}{n+k-1\choose
k}^{\frac{k-1}{2k}}.
\end{split}
\end{equation}
It follows from (\ref{Num-1}) that
$${n+k-1\choose k}\leq e^{k}\left(1+\frac{n}{k}\right)^{k},$$
which, together with (\ref{c-1.2}) and  (\ref{c-1.3}), implies  that

\begin{equation}\label{c-1.20}
\begin{split}
\sum_{\alpha}|(a_{\alpha}+b_{\alpha})z^{\alpha}|&\leq
f(0)+2\big(1-f(0)\big)\sum_{k=1}^{\infty}\kappa\|z\|_{\infty}^{k}(1+\varepsilon)^{k}
e^{\frac{k-1}{2}}\left(1+\frac{n}{k}\right)^{\frac{k-1}{2}}\\
&\leq
f(0)+2\big(1-f(0)\big)\sum_{k=1}^{\infty}\kappa\left(\|z\|_{\infty}\sqrt{e}(1+\varepsilon)\right)^{k}\left(1+\frac{n}{k}\right)^{\frac{k-1}{2}}.
\end{split}
\end{equation}

Let $\varepsilon\in(0,1/2)$. Set
$\|z\|_{\infty}=(1-2\varepsilon)\sqrt{(\log n)/n}$. Then, as in the
proof of \cite[Section 6]{BPS}, we obtain that
\begin{equation}\label{BPS-proof}
\sum_{k=1}^{\infty}\kappa\left(\|z\|_{\infty}\sqrt{e}(1+\varepsilon)\right)^{k}\left(1+\frac{n}{k}\right)^{\frac{k-1}{2}}
\leq \frac{1}{2}
\end{equation}
for large enough $n$.

Hence we conclude  from (\ref{c-1.20}) and (\ref{BPS-proof}) that
for large enough $n$, we have
\[\sum_{\alpha}|(a_{\alpha}+b_{\alpha})z^{\alpha}|\leq
 f(0)+(1-f(0))=1,\] which
yields that
$$\mathcal{R}_{P}(\mathbb{B}_{\ell_{\infty}^{n}})\geq(1-2\varepsilon)\sqrt{\frac{\log n}{n}}$$
for large enough $n$.

$\mathbf{Claim~ 2.}$
$\limsup_{n\rightarrow\infty}\frac{\mathcal{R}_{P}(\mathbb{B}_{\ell_{\infty}^{n}})}{\sqrt{\frac{\log
n}{n}}}\leq1$.

It is easy to know that if $f\in
\mathscr{H}_{1}(\mathbb{B}_{\ell_{\infty}^{n}})$ with $f(0)=0$, then
$f\in\mathscr{PH}_{+}(\mathbb{B}_{\ell_{\infty}^{n}})$. Moreover, we
observe that if $f\in
\mathscr{H}_{1}(\mathbb{B}_{\ell_{\infty}^{n}})$ with $f(0)\neq0,$
then $e^{-i\arg
f(0)}f\in\mathscr{PH}_{+}(\mathbb{B}_{\ell_{\infty}^{n}})$. Thus, we
can use a similar proof method as in \cite[Remark 1]{BK} to show
that Claim 2 is true. For the sake of completeness, we recall the
proof process. It follows from the Kahane-Salem-Zygmund inequality
(see \cite{BPS,BK}) that there exist a positive constant $C$ and
coefficients $(c_{\alpha})_{|\alpha|=k}$ with
$|c_{\alpha}|=k!/\alpha!$ for all $\alpha$ such that

\[
\sup_{z\in\mathbb{B}_{\ell_{\infty}^{n}}}\left|\sum_{|\alpha|=k}c_{\alpha}z^{\alpha}\right|\leq
C\sqrt{k\log k}(k!)^{\frac{1}{2}}n^{\frac{k+1}{2}},\] and
consequently,

\[
\mathcal{R}_{P}^{k}(\mathbb{B}_{\ell_{\infty}^{n}})n^{k}=\sum_{|\alpha|=k}|c_{\alpha}|\mathcal{R}_{P}^{k}(\mathbb{B}_{\ell_{\infty}^{n}})\leq
\sup_{z\in\mathbb{B}_{\ell_{\infty}^{n}}}\left|\sum_{|\alpha|=k}c_{\alpha}z^{\alpha}\right|\leq
C\sqrt{k\log k}(k!)^{\frac{1}{2}}n^{\frac{k+1}{2}}. \] Hence

\[ \mathcal{R}_{P}(\mathbb{B}_{\ell_{\infty}^{n}})\leq
C^{\frac{1}{k}}\left(\sqrt{k\log
k}\right)^{\frac{1}{k}}\frac{1}{\sqrt{n}}n^{\frac{1}{2k}}(k!)^{\frac{1}{2k}}.\]
If we take $k=[\log n] ~(n\geq3)$ and use Stirling's formula, then
we can get the desired result.

Therefore, this case follows from Claims 1 and 2.
 The proof of this theorem
is completed. \qed


\bigskip

\section*{Acknowledgments}
H. Hamada was partially supported by JSPS KAKENHI Grant Number JP19K03553.

\normalsize

\end{document}